\documentclass[a4paper,oneside,12pt]{amsart}
\usepackage[a4paper,margin=2.5cm]{geometry}
\usepackage[utf8]{inputenc}
\usepackage{lmodern,stackrel}
\usepackage[T1]{fontenc}
\usepackage{verbatim}
\usepackage{textcomp}
\usepackage[english]{babel}
\usepackage[pdftex]{hyperref}
\usepackage[pdftex]{color,graphicx}
\usepackage{amssymb}
\usepackage{upgreek}
\usepackage{graphicx, epstopdf}
\usepackage{dsfont}
\usepackage{mathrsfs,bbm}
\usepackage{marginnote}
\usepackage{tikz}
\usepackage{fourier}
\usepackage{mathtools}
\usepackage{amsmath}

\DeclarePairedDelimiterX\abs[1]{\lvert}{\rvert}{#1}

\def\llbracket{[\hspace{-.10em} [ }

\def\rrbracket{ ] \hspace{-.10em}]}
\renewcommand{\leq}{\leqslant}
\renewcommand{\geq}{\geqslant}

\def\build#1_#2^#3{\mathrel{
\mathop{\kern 0pt#1}\limits_{#2}^{#3}}}

\newcommand{\E}{\mathbb{E}}

\theoremstyle{plain}
\newtheorem{theorem}{Theorem}

\newtheorem{lemma}{Lemma}

\theoremstyle{definition}

\theoremstyle{remark}
\newtheorem{remark}{Remark}

\newcommand{\Z}{{\mathbb{Z}}}

\renewcommand{\P}{\mathbb{P}}

\title{The extinction of the contact process in a one-dimensional random environment with long-range interactions}

\author{Pablo A.\ Gomes}
\thanks{Departamento de Estat\'\i stica, Universidade de S\~ao Paulo, Rua do Mat\~ao 1010 CEP 05508-090 S\~ao Paulo-SP, Brazil. \url{pagomes@usp.br}}

\author{Marcelo R.\ Hilário}
\thanks{Departamento de Matem{\'a}tica, Universidade Federal de Minas Gerais, Av. Ant\^onio Carlos 6627 C.P. 702 CEP 30123-970 Belo Horizonte-MG, Brazil. \url{mhilario@mat.ufmg.br}}

\author{Bernardo N.\ B.\ de Lima}
\thanks{Departamento de Matem{\'a}tica, Universidade Federal de Minas Gerais, Av. Ant\^onio Carlos 6627 C.P. 702 CEP 30123-970 Belo Horizonte-MG, Brazil. \url{bnblima@mat.ufmg.br}}

\author{Thomas Mountford}
\thanks{Institute of Mathematics, \'Ecole Polytechnique Fédérale de Lausanne, Route Cantonale 1015 Lausanne, Switzerland. \url{thomas.mountford@epfl.ch}}

\usepackage{fancyhdr}
\begin{document}

\begin{abstract}
We study the contact process on the long-range percolation cluster on $\mathbb{Z}$ where each edge $\langle i,j \rangle$ is open with probability $|i-j|^{-s}$ for $s> 2$.
Using a renormalization procedure we apply Peierls-type argument to prove that the contact process dies out if the transmission rate is smaller than a critical threshold.
Our methods involve the control of crossing probabilities for percolation on randomly-stretched lattices as in \cite{Hilario2023}.
\end{abstract}

\maketitle


{\footnotesize Keywords: contact process; long-range percolation; anisotropic percolation; random environment }

\section{Introduction}

\subsection{Model and main results}

The aim of this article is to study the phase transition of the contact process on a random graph with long-range connections and unbounded degree distribution.

Let \( s > 0 \).
Starting from the complete graph on \(\mathbb{Z}\), we declare each unoriented edge \(\langle x,y \rangle\) to be open independently with probability \( |x-y|^{-s} \), and closed otherwise.
The subgraph consisting of only the open edges defines a random graph \( G = (\mathbb{Z}, \mathcal{E}) \), which is almost surely connected since it contains every edge of length one.
See Figure \ref{f:long_range}.

Given \( G \), we consider a contact process on \(\mathbb{Z}\) in which infections can spread exclusively along open edges.  
Starting from a given nonempty set \( A \subseteq \mathbb{Z} \) of infected sites, the system evolves as follows:  
\begin{itemize}
    \item At rate \( \lambda \), each infected site \( x \) transmits the infection to each site \( y \) such that \( \langle x, y \rangle \in \mathcal{E} \).  
    \item Infected sites recover at rate \( 1 \).
\end{itemize}

Let \( \zeta^{A,G}_t \) denote the resulting Markov process in \( \{0,1\}^{\mathbb{Z}} \), where $1$ stands for infected and $0$ for healthy.
We may omit the dependence on \( G \) and simply write \( \zeta^A_t \).  
We define the critical infection rate $\lambda_c$ as the threshold value for survival.
That is,
\[
\lambda_c(\mathcal{E}) = \inf\big\{\lambda > 0 \colon {P}_{\lambda}^{\mathcal{E}}\big(\zeta^{\{o\}}_t \neq \varnothing \text{ for every } t \geq 0\big) > 0\big\},
\]
where $P^{\mathcal{E}}_{\lambda}$ stands for the law of the contact process on $\mathcal{E}$.
Standard ergodicity arguments imply that \( \lambda_c (\mathcal{E})\) does not depend on the realization of $\mathcal{E}$, almost surely; hence we write simply $\lambda_c$.
Since \( G \) almost surely contains a copy of \(\mathbb{Z}\), we conclude that \( \lambda_c < \infty \); this was proved originally by~\cite{Harris74} (see Section 1 of Chapter VI of~\cite{Liggett1985} or Chapter 6 of~\cite{Grimmett2010} for a more comprehensive approach).

The interesting question is to determine whether a subcritical phase exists, that is, whether $\lambda_c >0$.
When the underlying random graph has unbounded degree distribution, this is often a challenging question due to the presence of vertices with atypically high degrees that are able to sustain infections for extended periods (see Section \ref{s:related_works} for a discussion on related works on this problem).

For $G = (\mathbb{Z},\mathcal{E})$, when $s\leq 1$, each vertex has infinite degree a.s., therefore the contact processes survives as soon as $\lambda >0$, that is, $\lambda_c=0$.
In \cite{HaoCan_15}, it is shown that if \( s \) is sufficiently large, then \( \lambda_c > 0 \), and it is conjectured that this should hold as soon as \( s > 2 \).  
Our main result resolves this conjecture.

\begin{theorem}
\label{theo1}  
If $s>2$ then the critical infection rate for the contact process on the graph $G=(\mathbb{Z},\mathcal{E})$ satisfies  \( \lambda_c >0 \).
\end{theorem}

\begin{center}
\begin{figure}
\begin{tikzpicture}[scale=0.27, every node/.style={circle, fill=gray!70, inner sep=1pt}]

    \def\n{45}
    
    \foreach \x in {1,...,\n} {
        \node (s\x) at (\x, 0) {};
    }
    \foreach \x in {1,...,44} {
        \draw (s\x) -- (s\the\numexpr\x+1);
    }

    \def\angleShort{40}  
    \def\angleMedium{50} 
    \def\angleLong{55}   

    \draw[thick] (s1) to[out=\angleLong, in=180-\angleLong] (s9);
    \draw[thick] (s1) to[out=\angleMedium, in=180-\angleMedium] (s8);
    \draw[thick] (s2) to[out=\angleMedium, in=180-\angleMedium] (s7);
    \draw[thick] (s3) to[out=\angleShort, in=180-\angleShort] (s6);
    \draw[thick] (s4) to[out=\angleMedium, in=180-\angleMedium] (s8);

    \draw[thick] (s11) to[out=\angleMedium, in=180-\angleMedium] (s15);
    \draw[thick] (s11) to[out=\angleShort, in=180-\angleShort] (s13);
    \draw[thick] (s13) to[out=\angleShort, in=180-\angleShort] (s19);
    \draw[thick] (s13) to[out=\angleShort, in=180-\angleShort] (s18);
    \draw[thick] (s13) to[out=\angleShort, in=180-\angleShort] (s16);
    \draw[thick] (s16) to[out=\angleShort, in=180-\angleShort] (s18);

    \draw[thick] (s21) to [out=\angleShort, in=180-\angleShort] (s24) ;
    \draw[thick] (s21) to [out=\angleLong, in=180-\angleLong] (s29) ;
    \draw[thick] (s24) to[out=\angleMedium, in=180-\angleMedium] (s29);
    \draw[thick] (s24) to[out=\angleShort, in=180-\angleShort] (s27);
    \draw[thick] (s25) to[out=\angleShort, in=180-\angleShort] (s28);
    \draw[thick] (s29) to[out=\angleMedium, in=180-\angleMedium] (s34);
    \draw[thick] (s31) to[out=\angleMedium, in=180-\angleMedium] (s33);
    \draw[thick] (s32) to[out=\angleShort, in=180-\angleShort] (s34);

    \draw[thick] (s36) to[out=\angleMedium, in=180-\angleMedium] (s39);
    \draw[thick] (s37) to[out=\angleMedium, in=180-\angleMedium] (s41);
    \draw[thick] (s41) to[out=\angleMedium, in=180-\angleMedium] (s45);
    \draw[thick] (s42) to[out=\angleShort, in=180-\angleShort] (s44);
    \draw[thick] (s42) to[out=\angleShort, in=180-\angleShort] (s45);
    
    \foreach \x in {10, 20, 35} {
        \node[red] at (\x,0) [circle, fill=red, inner sep=1pt] {};
    }
    \foreach \x in {1, 9, 11, 19, 21, 29, 34, 36, 41} {
        \node at (\x,0) [circle, fill=black, inner sep=1pt] {};
    }
    
\end{tikzpicture}
\caption{A realization of the long-range percolation. 
Assumption $s>2$ guarantees existence of infinitely many cut-points (black and red). Strong cut-points (red) are surrounded by cut-points.}
\label{f:long_range}
\end{figure}
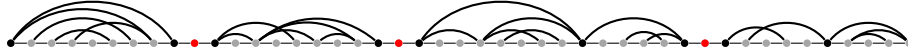
\end{center}

\begin{remark}
During preparation of the final version of this manuscript, we became aware of the recent work \cite{Jahnel25}, in which the authors also prove the same conjecture, among other results.  
Both works were developed independently. Although they share some similarities, the proofs rely on different techniques.  
\end{remark}

While our results establish the existence of a phase transition for \( s > 2 \), the behavior of the model in the intermediate regime \( 1 < s \leq 2 \) remains an open question. 
Determining whether a nontrivial critical infection rate exists in this range would provide a more complete understanding of the model, analogous to the well-established phase diagram for other one-dimensional systems with long-range interactions of order \( |i - j|^{-s} \) such as the Ising model \cite{Dyson69, Frolich84}  and Bernoulli percolation \cite{Newman86, Aizenman86}.

\subsection{Model and definitions}

The contact process introduced by Harris in \cite{Harris74} is an interacting particle system often regarded as a model for the spread of an infection (or information) in a graph.

We will study the contact process on $G$ through its graphical representation.
We now describe the main ideas involved and refer the reader to the classic work by Harris \cite{Harris78} for a detailed and formal construction.

The construction is based on two independent families of Poisson processes $(\mathcal{I}_e \colon e\in \mathcal{E})$ with rate $\lambda$ and $(\mathcal{R}_x \colon x\in \mathbb{Z})$ with rate $1$.
The infection evolves on the space-time structure $\mathbb{Z} \times \mathbb{R}_+$, where the vertical axis represents time. For each vertex $x \in \mathbb{Z}$ and each time $t \in \mathcal{R}_x$, we place a lozenge at the point $(x, t)$, marking a recovery event that halts the infection at $x$. Additionally, for each edge $e = \langle x, y \rangle \in \mathcal{E}$ and each time $t \in \mathcal{I}_e$, we draw a horizontal double arrow between $(x, t)$ and $(y, t)$, indicating a potential transmission of the infection from $x$ to $y$ and from $y$ to $x$.

Consider an initial set of infected vertices $A \subset \mathbb{Z}$ at time 0. 
Infection paths start at points in \( A \times \{0\} \) and can only move upward along vertical lines (representing the passage of time) until they encounter a lozenge, at which point the path terminates (indicating recovery).
Alternatively, a path can jump horizontally from \( (x, t) \) to \( (y, t) \) by crossing a double arrow associated with the edge \( \langle x, y \rangle \), enabling the infection to spread to vertex \( y \) at time \( t \).
An illustration of the graphical representation, including a sample infection path starting from the origin \( o \in \mathbb{Z} \), is provided in Figure~\ref{f:graph_rep}.

\begin{center}
\begin{figure}[htb!]
\begin{tikzpicture}[scale=0.25, every node/.style={circle, fill, inner sep=.75pt}]
    \def\n{51}
    
    \fill[gray!10] (28.5,0) rectangle (51.5,22);
    
    \foreach \x in {29,...,\n} {
        \node (s\x) at (\x, 0) {};
    }
    \foreach \x in {29,...,50} {
        \draw (s\x) -- (s\the\numexpr\x+1);
    }

    \def\angleShort{40}  
    \def\angleMedium{50} 
    \def\angleLong{35}   

    \draw[thick] (s29) to[out=\angleShort, in=190-\angleShort] (s31);
    \draw[thick] (s31) to[out=\angleMedium, in=180-\angleMedium] (s36);
    \draw[thick] (s32) to[out=\angleShort, in=180-\angleShort] (s35);
    \draw[thick] (s34) to[out=\angleMedium, in=180-\angleMedium] (s39);
    \draw[thick] (s41) to[out=\angleLong, in=180-\angleLong] (s49);
    \draw[thick] (s42) to[out=\angleShort, in=180-\angleShort] (s45);
    \draw[thick] (s46) to[out=\angleMedium, in=180-\angleMedium] (s48);
    \draw[thick] (s48) to[out=\angleShort, in=190-\angleShort] (s51);

    \foreach \x in {40} {
        \node[red] at (\x,0) [circle, fill=red, inner sep=1pt] {};
    }
    \node[below, draw=none, fill=none] at (40,0) {$o$};

    \draw[thick, red] 
        (40,0) -- (40,8) 
        (40,6) -- (41,6) -- (41,12) 
        (40,2) -- (39,2) 
        (39,2) -- (39,5) 
        (39,2) -- (39,3) 
        (39,3) -- (34,3) -- (34,11) 
        (41,9) -- (49,9) -- (49,15);

    \draw[thick, <->, red] (40,6) -- (41,6);
    \draw[thick, <->, red] (40,2) -- (39,2);
    \draw[thick, <->, red] (39,3) -- (34,3);
    \draw[thick, <->, red] (41,9) -- (49,9);

    \draw[thin, <->] (29,1) -- (31,1);    
    \draw[thin, <->] (49,1) -- (50,1);    
    \draw[thin, <->] (31,18.5) -- (36,18.5); 
    \draw[thin, <->] (46,7.5) -- (48,7.5); 
    \draw[thin, <->] (30,9.5) -- (31,9.5); 
    \draw[thin, <->] (32,13) -- (35,13);   
    \draw[thin, <->] (35,21) -- (36,21);   
    \draw[thin, <->] (42,13) -- (45,13);   
    \draw[thin, <->] (42,4) -- (45,4);   
    \draw[thin, <->] (48,17) -- (51,17);   
    \draw[thin, <->] (45,15) -- (46,15);   
    \draw[thin, <->] (42,12) -- (43,12);   
    \draw[thin, <->] (35,16) -- (36,16);   
    \draw[thin, <->] (41,18) -- (49,18);

    \foreach \h in {29,...,51} {
        \draw[dotted, thin] (\h,0)--(\h,22);
    }

    \foreach \h/\th in {29/12, 31/15, 34/11, 35/17, 37/21, 
                        40/15, 41/12, 42/20, 46/19, 49/15} {
        \draw[blue!50, fill] (\h,\th-0.5) -- (\h+0.5,\th) -- (\h,\th+0.5) -- (\h-0.5,\th) -- cycle;
    }
    \foreach \h/\th in {33/3, 35/7, 39/5, 40/8,
                        43/7, 47/4, 48/5, 51/10} {
        \draw[blue!50, fill] (\h,\th-0.5) -- (\h+0.5,\th) -- (\h,\th+0.5) -- (\h-0.5,\th) -- cycle;
        \draw[dotted, thin] (\h,0)--(\h,\th);
    }

    \clip (29,0) rectangle (51,22);
\end{tikzpicture}
\caption{Graphical representation of the contact process on a random one-dimensional graph with long-range connections. Lozenges (blue) represent recovery marks, while horizontal double arrows (gray and red) indicate potential transmission marks. The red path illustrates the spread of infection starting from the vertex $o$.}
\label{f:graph_rep}
\end{figure}
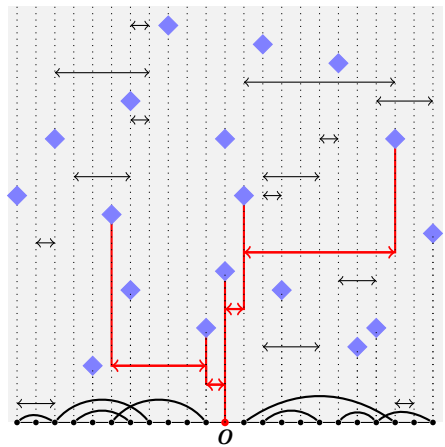
\end{center}

\subsection{Related works}
\label{s:related_works}

The contact process on graphs like the $d$-dimensional Euclidean lattice and the $d$-ary regular tree is a well-understood model (see \cite{Liggett1985, Liggett1999} for comprehensive accounts).

Over the past two decades, research has increasingly focused on the contact process on random graphs.
To the best of our knowledge, sufficient conditions for the existence of a non-trivial subcritical phase for the contact process on random graphs with unbounded degree distributions were first established by Ménard and Singh \cite{Menard_16}. Their conditions apply to graphs such as Delaunay triangulations and random geometric graphs. 
In \cite{HaoCan_15}, Can verified these conditions for the long-range percolation graph when $s>102$. 
More recently, Jahnel et al. \cite{Jahnel25} established the existence of a subcritical phase for all $s>2$, independently of our work.
Their proof relies on a coupling with a random walk on a random environment while ours, as is classic in percolation, is based on a Peierls argument.

In \cite{Jahnel25}, the authors also study the phase transition for the contact process on one-dimen\-sional random geometric graphs with i.i.d.\ random radii.
They show that a phase transition takes place if, and only if, the radius distribution has finite mean. The contact process on the dynamic version of this random graph, where the radii update independently at rate one was studied in \cite{GL2022}.

Many other authors have considered different kinds of random graphs, mainly Galton-Watson trees and the configuration model.
We now list some of the results in the field and refer the reader to \cite{Valesin2024} for a gentle and comprehensive account on the topic.

In \cite{Pemantle01}, Pemantle and Stacey studied the contact process on Galton-Watson trees with bounded-degree offspring distributions. 
They identified examples where the critical points for local and global survival differ, leading to a double phase transition, as observed in regular $d$-ary trees. 
However, the existence of a subcritical phase for Galton-Watson trees with unbounded-degree offspring remained an open question.
Huang and Durrett \cite{Huang2020} showed that if the offspring distribution is subexponential, no subcritical phase exists.
In \cite{Bhamidi2021}, Bhamidi et al.\ established that a subcritical phase exists if and only if the offspring distribution has exponential tails.

The contact process on random graphs generated by the configuration model on $n$ vertices with a power-law degree distribution (with exponent greater than 3) was studied by Chatterjee and Durrett \cite{ChatterjeeDurrett2009}. 
As these graphs are finite, the contact process eventually dies out. 
However, they proved that no phase transition occurs, meaning that for any positive infection parameter $\lambda$, the process survives for an exceptionally long time. 
Specifically, the extinction time grows at least as a stretched exponential function of the number of vertices. 
Mountford et al. \cite{Mountford2016} improved this result, showing that the survival time is exponential in the number of vertices. 
Additionally, Mountford et al.\ \cite{MountfordValesinYao2013} established sharp bounds for the density of infected sites.
For the particular case when the degrees equal some constant $d$, also called random $d$-regular graph, Mourrat and Valesin \cite{Mourrat2016} and Lalley and Su \cite{Lalley2017} demonstrated the existence of a phase transition. 
The latter work also established a cutoff for the density of infected sites.
Bhamidi et al.\ \cite{Bhamidi2021} showed that a phase transition occurs if and only if the degree distribution has an exponential moment.

\subsection{Summary of the paper}

The proof of Theorem \ref{theo1} is based on a Peierls-type argument for a percolation model obtained after we perform a one-step renormalization, which we summarize next.
\medskip

\noindent {\bf Divide the graph into intervals bounded by cut-points: }
We start by partitioning the underlying random graph $G=(\mathbb{Z}, \mathcal{E})$ into infinitely many intervals $V_k$ ($k\in\mathbb{Z}$) whose endpoints are cut-points.
These are vertices $v\in\mathbb{Z}$ for which no open edge $\langle x,y \rangle$ exists with $x < v < y$.
This implies that there are no edges linking sites from intervals $V_k$ and $V_j$ with $k \neq j$.
The existence of a doubly infinite sequence of cut-points is ensured by the assumption that $s > 2$.
This is the done in Section \ref{s:cut_points}.
\medskip

\noindent {\bf Controlling the number of vertices and edges within  intervals:} 
We then show that the number of vertices within the intervals $V_k$ behaves as a sequence of independent random variables with finite $(1+\varepsilon)$-moment, for some sufficiently small $\varepsilon>0$ (Lemma \ref{LemmaV}).
The same holds for the number of edges connecting pairs of vertices in $V_k$ (Lemma \ref{l:expectation_ek}).
\medskip

\noindent {\bf Renormalizing to a box lattice:}
Next, we partition the time axis into intervals of length $T > 0$.
This leads to a tiling of the space-time domain $\mathbb{Z} \times \mathbb{R}_+$ into boxes \(V_k \times [jT, (j+1)T)\) that project horizontally onto $V_k$ and vertically onto strips of length $T$.

A box is said to be good if:
\begin{itemize}
\item every vertex $x \in V_k$ has a recovery mark in $[jT, (j+1)T)$;
\item for every pair $x, y \in V_k$, there are no transmission attempts along the edge $\langle x, y \rangle$ in $[jT, (j+1)T)$.
\end{itemize}

The construction is carefully made to ensure that the status of each box is independent.
Moreover, the probability that a box \(V_k \times [jT, (j+1)T)\) is good is bounded below by $p^{\phi_k}$, where $(\phi_k)_{k\in\mathbb{Z}}$ are i.i.d.\ random variables with finite $(1+\varepsilon)$-moment, and $p = p(\lambda, T)$ can be made arbitrarily close to $1$ by choosing $T$ large enough and $\lambda$ small.

This leads to a site percolation model on the renormalized lattice $\mathbb{Z} \times \mathbb{Z}_+$, where the probability that a site is open depends on the column it belongs to.

This is the content of Section \ref{s:renormalization}.
The reader is invited to consult Figure \ref{f:renormalization_boxes} for an illustration of a good box.
\medskip

\noindent {\bf Peierls argument:} A Peierls-type argument shows that the contact process dies out if there exists a semicircuit of good boxes surrounding the origin.
In fact, since the left and right boundaries of each box project onto cut-points, if the infection ever enters a good box \(V_k \times [jT, (j+1)T)\), it could only propagate to neighboring boxes $V_{k-1} \times [jT, (j+1)T)$, \(V_{k+1} \times [jT, (j+1)T)\), or $V_k \times [(j+1)T, (j+2)T)$ passing through their common boundaries.
However that propagation if forbidden if the neighboring boxes are also good (see Figure \ref{f:renormalization_boxes}).
Therefore, an infection starting at the origin  cannot reach the exterior of the semicircuit.

Our task now reduces to showing that, for $p$ sufficiently large, a semicircuit of open sites surrounding the origin exists almost surely for the above site percolation.
If we were dealing with a percolation model with weak dependencies, standard techniques could be used to establish the existence of such circuits.
However, since the parameters vary across columns, correlations along the vertical axis do not decay with distance (under the annealed law), complicating the implementation of a Peierls-type argument.
\medskip

\noindent {\bf Comparison with bond percolation on a randomly stretched lattice:} 
It turns out that the probability of finding semicircuits in the above site percolation model can be compared with that of a bond percolation model studied in \cite{Hilario2023}.

It is defined as follows: Let $\Lambda = \{\xi_k \colon k \in \mathbb{Z}\}$ be a sequence of i.i.d.\ positive random variables. 
Conditional on $\Lambda$, each edge $e$ of the lattice $\mathbb{Z}\times \mathbb{Z}_+$ is declared open independently with probability
\begin{equation}
\label{e:def_perc_stret}
\rho_e = \left\{
\begin{array}{ll}
\rho, & \text{if } e = \{(i, j), (i, j+1)\} \text{ for some } i, j, \\
\rho^{\xi_{i+1}}, & \text{if } e = \{(i, j), (i+1, j)\} \text{ for some } i, j.
\end{array}
\right.
\end{equation}
We will write $\mathbb{P}^{\Lambda}_{\rho}$ for the corresponding probability measure induced in $\{0,1\}^{\mathcal{E}(\mathbb{Z}^d)}$

When the sequence $(\xi_n)_n$ consists of independent geometric random variables, the model is reminiscent of the one introduced in  \cite{JMP} and studied further in \cite{Hoffman_2005, Hilario2023+}.
In our setting, however, the marginal distribution of $\xi_n$ is no longer geometric, requiring a different analytical approach.
To handle this, we rely on the results developed in \cite{Hilario2023}.
There the authors derive estimates for the probability of crossing rectangles and apply them to show that for $p$ large enough the connected component containing the origin is infinite with positive probability. 
Here we use crossings of rectangles in order to build a semicircuit around the origin with high probability.
This is done in Sections \ref{sec:crossing} and \ref{sec:stretched}.

\section*{Acknowledgements} 
The authors thank Daniel Valesin for valuable discussions, and two referees for comments, suggestions and corrections.
P.G.\ acknowledges the hospitality of EPFL, where  discussions with Prof.~Thomas Mountford were initially held. The research of P.G. is partially supported by Fundação de Amparo à Pesquisa do Estado de São Paulo grants	21/10170-7 and 23/13453-5.
The research of M.H.\ is partially supported by Fundação de Amparo à Pesquisa do Estado de Minas Gerais grant APQ-01214-21 (Universal), Conselho Nacional de Desenvolvimento Científico e Tecnológico grant 312566/2023-9 (Produtividade em Pesquisa) and grant 406001/2021-9 (Universal).
B.N.B.L. is partially supported by Conselho Nacional de Desenvolvimento Científico e Tecnológico grant 315861/2023-1.

\section{Control of the environment, cut-points}
\label{s:cut_points}

Our running assumption throughout the article will be that $s>2$.
It implies that the expected number of edges that connect sites on opposite sides of the origin is finite almost surely.
In fact,
\begin{equation}
\label{e:over_origin}
\sum_{i \in \Z_-} \sum_{j \in \Z_+} \P\big( \langle i, j\rangle \text{ is open} \big)= \sum_{i \in \Z_-} \sum_{j \in \Z_+} \frac{1}{|i-j|^s} < \infty.
\end{equation}
We define
\begin{eqnarray} 
x_{+} &=& \sup\big\{ j \in \Z_+ \colon \textrm{$\langle i, j \rangle$ is open for some $i \in \Z_-$}\big\}, \\
x_{-} &=& \inf\big\{ i \in \Z_- \colon \text{ $\langle i, j \rangle$ is open for some $j \in \Z_+$}\big\},
\end{eqnarray}
with the convention that $\sup(\varnothing) = \inf(\varnothing) =0$.
It follows from \eqref{e:over_origin} that both $x_{+}$ and $x_{-}$ are finite, almost surely.
Notice that, in order to determine $x_{-}$ and $x_{+}$ one only needs to reveal the states of edges connecting sites on opposite sides of the origin.

A cut-point is a vertex $x \in \Z$ such that, for every pair of positive integers $i , j \in \Z_+$, the edge $\langle x - i , x + j\rangle$ is closed.
We have
\begin{equation}
\begin{split}
\P (\text{$o$ is a cut-point}) & = \prod_{i \geq 1} \prod_{j \geq 1} \P\big(\langle - i ,  j\rangle \textrm{ is closed}\big) =  \prod_{i \geq 1} \prod_{j \geq 1} \left(1 - \frac{1}{|j+i|^s} \right) \\
& \geq  \tfrac{1}{2}\exp\left(- \sum_{i \geq 1} \sum_{j \geq 1}\frac{1}{|j+i|^s} \right) > 0,
\end{split}
\end{equation}
where we used the assumption $s>2$ in the last inequality.
By translation invariance there exist infinitely many cut-points on both sides of the origin, almost surely.
We look for those lying to the right of $x_{+}$ exploring the state of edges in the positive half-line as follows.
Define
\begin{equation*} 
x_{1} = \min\{ x > x_+ \colon \textrm{ $\langle i,j  \rangle$  is closed for all $0 \leq i < x <  j$}\}
\end{equation*} 
and inductively for $n\geq 1$,
\begin{equation*} 
x_{n+1} = \min\big\{ x > x_n \colon \textrm{ $\langle i, j \rangle$ is closed for all $x_n \leq i < x <  j$}\big\},
\end{equation*}
Analogously, we can perform the exploration to the left of $x_{-}$ defining
\begin{equation*} 
x_{-1} = \max\big\{ x < x_- \colon \textrm{ $\langle i,j  \rangle$ is closed for all  $i < x <  j \leq 0$}\big\}
\end{equation*} 
and inductively for $n\geq 1$,
\begin{equation*} 
x_{-(n+1)} = \max\big\{ x < x_{-n} \colon \text{$\langle i, j \rangle$ is closed for all  $i < x <  j \leq x_{-n}$}\big \},
\end{equation*}
This yields a bi-infinite sequence $(x_k)_{k \in \mathbb{Z} \setminus \{0\}}$ of cut-points.

We now define, for each $n \geq 1$,
\begin{equation}\label{xi}
\theta_n = x_{n+1} - x_{n} \quad \textrm{ and } \quad \theta_{-n} = x_{-n} - x_{-(n+1)}.
\end{equation} 
Given the above construction of the sequence $\{x_k, ~ k \in \mathbb{Z} \setminus \{0\}\}$, one can readily check that $\{\theta_k \colon k\in \mathbb{Z} \setminus \{0\}\}$ forms an i.i.d.\ sequence.
We denote by $\theta$ a random variable with the same distribution as any of the variables in that sequence.
The following result shows that $\theta$ has some moment of order larger than one.
\begin{lemma}
\label{Lemmaxi}
For every $0 \leq \delta < s-2$,
\begin{equation*}
\E\left[\theta^{1+\delta}\right] < \infty.
\end{equation*}
\end{lemma}

For technical reasons, it will be convenient to extract a subsequence from  $(x_k)_{k\in \mathbb{Z}^\ast}$ where each element is also surrounded by adjacent cut-points.
For $k \geq 2$, an element $x_k$ is called a strong cut-point if $\theta_{k-1} = \theta_{k} = 1$. 
Similarly, for $k \leq  -2$, the cut-point $x_k$  is called strong if $\theta_{k+1} = \theta_{k} = 1$. 
Define  
\begin{equation*} 
\begin{split}
n_{1} &= \min\{ n > 1 \colon \theta_{2n-1} = \theta_{2n} = 1 \}
\quad \textrm{and, for $k\geq 1$,} \\
n_{k+1} &= \min\{ n > n_{k} \colon \theta_{2n-1} = \theta_{2n} = 1 \}.
\end{split}
\end{equation*}
Analogously, on the left, define
\begin{equation*} 
\begin{split}
n_{-1} &= \max\{ n < 1 \colon \theta_{2n+1} = \theta_{2n} = 1  \}
\quad \textrm{and, for $k\geq 1$,} \\
n_{-(k+1)} &= \max\{ n < n_{-k}  \colon \theta_{2n+1} = \theta_{2n} = 1 \}.
\end{split}
\end{equation*}
(It is not difficult to verify that the $min$ and $max$ run over nonempty sets and are well-defined.)

We now will pick every other strong cut-point. 
Define, for each $k \in \mathbb{Z} \setminus \{0\}$, 
\begin{equation}
v_k = x_{2n_{k}}.
\end{equation}
This gives rise to a sequence $\{ v_k, ~ k \in \mathbb{Z} \setminus \{0\}\}$ of strong cut-points spanning $\mathbb{Z}$. 

Analogously to \eqref{xi}, we define
\begin{equation}
\label{eta}
\eta_k = v_{k+1} - v_{k} \quad \textrm{ and } \quad \eta_{-k} = v_{-k} - v_{-(k+1)}.
\end{equation} 
As before, we have that $\{\eta_k \colon k\in \mathbb{Z} \setminus \{0\}\}$ forms an i.i.d.\ sequence.
We denote by $\eta$ a random variable with the common distribution of $\{\eta_k \colon k\in \mathbb{Z} \setminus \{0\}\}$.
The following result is an extension of Lemma~\ref{Lemmaxi}. 
\begin{lemma}
\label{Lemmaeta}
For every $0 < \delta < s-2$,
\begin{equation*}
\E\left[\eta^{1+\delta}\right] < \infty.
\end{equation*}
\end{lemma} 

Given the new sequence of strong cut-points $\{ v_k, ~ k \in \mathbb{Z} \setminus \{0\}\}$, we will partition $\mathbb{Z}$ into intervals $V_{k}$, where 
\begin{equation}
\label{e:def_V_0}
V_0 = \left\{x \in \Z \colon v_{-1} < x < v_{1} \right\};
\end{equation}
for each $k > 0$,
\begin{equation}
\label{e:def_V_k_pos}
V_{2k-1} = \{v_k\}, \quad V_{2k} = \left\{x \in \Z \colon v_{k} < x < v_{k+1} \right\};
\end{equation}
and, for each $k < 0$,
\begin{equation}
\label{e:def_V_k_neg}
    V_{2k+1} = \{v_k\}, \quad V_{2k} = \left\{x \in \Z \colon v_{k - 1} < x < v_{k} \right\}.
\end{equation}
See Figure \ref{f:cut_points_and_intervals} for an illustration of these sets.

\begin{center}
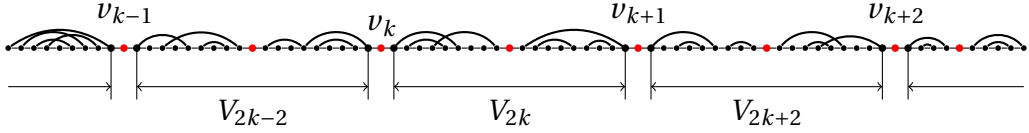
\begin{figure}[htb!]
\begin{tikzpicture}[scale=0.17, every node/.style={circle, fill, inner sep=.75pt}]
    \def\n{80}
    
    \foreach \x in {1,...,\n} {
        \node (s\x) at (\x, 0) {};
    }
    \foreach \x in {1,...,79} {
        \draw (s\x) -- (s\the\numexpr\x+1);
    }

    \def\angleShort{40}  
    \def\angleMedium{50} 
    \def\angleLong{35}   

    \draw[thick] (s1) to[out=\angleLong, in=180-\angleLong] (s9);
    \draw[thick] (s2) to[out=\angleMedium, in=180-\angleMedium] (s7);
    \draw[thick] (s3) to[out=\angleShort, in=180-\angleShort] (s6);
    \draw[thick] (s4) to[out=\angleMedium, in=180-\angleMedium] (s8);
    \draw[thick] (s11) to[out=\angleMedium, in=180-\angleMedium] (s15);
    \draw[thick] (s13) to[out=\angleShort, in=180-\angleShort] (s19);
    \draw[thick] (s16) to[out=\angleShort, in=180-\angleShort] (s18);
    \draw[thick] (s21) to[out=\angleShort, in=180-\angleShort] (s24);
    \draw[thick] (s24) to[out=\angleMedium, in=180-\angleMedium] (s29);
    \draw[thick] (s25) to[out=\angleShort, in=180-\angleShort] (s28);
    \draw[thick] (s31) to[out=\angleMedium, in=180-\angleMedium] (s36);
    \draw[thick] (s32) to[out=\angleShort, in=180-\angleShort] (s35);
    \draw[thick] (s34) to[out=\angleMedium, in=180-\angleMedium] (s39);
    \draw[thick] (s41) to[out=\angleLong, in=180-\angleLong] (s49);
    \draw[thick] (s42) to[out=\angleShort, in=180-\angleShort] (s45);
    \draw[thick] (s46) to[out=\angleMedium, in=180-\angleMedium] (s48);
    \draw[thick] (s51) to[out=\angleMedium, in=180-\angleMedium] (s56);
    \draw[thick] (s53) to[out=\angleShort, in=180-\angleShort] (s55);
    \draw[thick] (s57) to[out=\angleShort, in=180-\angleShort] (s59);
    \draw[thick] (s61) to[out=\angleMedium, in=180-\angleMedium] (s66);
    \draw[thick] (s63) to[out=\angleShort, in=180-\angleShort] (s65);
    \draw[thick] (s64) to[out=\angleLong, in=180-\angleLong] (s69);
    \draw[thick] (s71) to[out=\angleMedium, in=180-\angleMedium] (s74);
    \draw[thick] (s72) to[out=\angleShort, in=180-\angleShort] (s73);
    \draw[thick] (s76) to[out=\angleMedium, in=180-\angleMedium] (s80);
    \draw[thick] (s77) to[out=\angleShort, in=180-\angleShort] (s79);

    \foreach \x in {10, 20, 30, 40, 50, 60, 70, 75} {
        \node[red] at (\x,0) [circle, fill=red, inner sep=1pt] {};
    }
    \node[above, draw=none, fill=none] at (10,0) {$v_{k-1}$};
    \node[above, draw=none, fill=none] at (30,0) {$v_{k}$};
    \node[above, draw=none, fill=none] at (50,0) {$v_{k+1}$};
    \node[above, draw=none, fill=none] at (70,0) {$v_{k+2}$};

    \node[ draw=none, fill=none] at (20,-5) {$V_{2k-2}$};
    \node[ draw=none, fill=none] at (40,-5) {$V_{2k}$};
    \node[ draw=none, fill=none] at (60,-5) {$V_{2k+2}$};


    \foreach \x in {10, 30, 50, 70} {
        \draw[thin] (\x+1,0) -- (\x+1,-4);
        \draw[thin] (\x-1,0) -- (\x-1,-4);
        \node at (\x-1,0) [circle, fill=black, inner sep=1pt] {};
        \node at (\x+1,0) [circle, fill=black, inner sep=1pt] {};
    }

    \foreach \x in {10, 30, 50} {
        \draw[thin, <->] (\x+1,-3) -- (\x+19,-3);
        }
    \draw[thin, ->] (1,-3) -- (9,-3);
    \draw[thin, <-] (71,-3) -- (80,-3);
 
\end{tikzpicture}
\caption{
The representation of the strong cut-points $v_k$ and the corresponding sets $V_{2k}$.}
\label{f:cut_points_and_intervals}
\end{figure}
\end{center}

The random variables $|V_k|$ are independent, but their laws depend on $k$.
For odd values of $k$, the intervals are degenerate and $|V_k|$=1.
When $k \neq 0$ is even, $|V_k| = \eta_{k/2} - 1$, hence, as a consequence of Lemma~\ref{Lemmaeta} we have 
\begin{lemma}
\label{LemmaV}
For every $0 < \delta <s-2$ and $k \neq 0$
\begin{equation}
\label{moment}
\E\Big[|V_k|^{1+\delta}\Big] < \infty.
\end{equation}
\end{lemma}
\begin{remark}
\label{r:law_V_0}
    Note that the above discussion does not include $V_0$. 
    In fact, as seen in \eqref{e:def_V_0}, \eqref{e:def_V_k_pos} and \eqref{e:def_V_k_neg} the law of $V_0$ differs from the others.
\end{remark}

We will also partition the set of open edges according to the interval $V_k$ where an edge originates.
Let $\mathcal{E} = \cup_{k \in \Z} \mathcal{E}_{k}$, where 
\[
\mathcal{E}_0 = \big\{ \langle x,y \rangle \in \mathcal{E} \colon \{x,y\} \cap V_0 \neq \varnothing \big\},
\]
and
\begin{equation*}
\mathcal{E}_{k} = \begin{cases} 
\bigl\{ \langle x,y \rangle \in \mathcal{E} \colon \max\{x,y\} \in V_{k} \bigr\}, &~\textrm{ if } k \in \Z_-,\\
\big\{ \langle x,y \rangle \in \mathcal{E} \colon \min\{x,y\} \in V_{k} \big\}, &~\textrm{ if } k \in \Z_+.
\end{cases}
\end{equation*}
When $k$ is odd, $|\mathcal{E}_k| =1$. 
In fact, since $v_k$ is a strong cut-point, either $\mathcal{E}_{2k-1} = \{\langle v_k, v_k + 1 \rangle \}$ or $\mathcal{E}_{2k+1} = \{\langle v_k - 1, v_k \rangle \}$ according to whether $k$ is  positive or negative.
The random variables $\{|\mathcal{E}_k|, ~ k\in 2\mathbb{Z} \setminus \{0\}  \}$ are i.i.d.\
The next result shows that they have finite moments of order greater than one. 
\begin{lemma}
\label{l:expectation_ek} 
If $\varepsilon>0$ is sufficiently small, then for every $k\neq 0$, 
\[\E\Big[\left\vert\mathcal{E}_k \right\vert^{1+\varepsilon}\Big] < \infty.\]
\end{lemma}

\begin{remark}
The law of $|\mathcal{E}_0|$ differs from the others because the same holds for $V_0$ (see Remark \ref{r:law_V_0}).
\end{remark}

The remainder of this section is dedicated to the proof of Lemmas~\ref{Lemmaxi}, \ref{Lemmaeta}, and \ref{l:expectation_ek}.
\medskip

For the proof of Lemmas~\ref{Lemmaxi}, it will be convenient to first replace the original sequence $(x_k)_{k \in \mathbb{Z} \setminus \{0\}}$ with a modified version $(x_k^L)_{k \in \mathbb{Z} \setminus \{0\}}$, which is constructed in a similar manner but disregards edges shorter than $L$.
That is, for any fixed positive integer $L$, we set
\begin{equation*}
x_{1}^L = \min\{ x > x_+ \colon \textrm{ $\langle i,j  \rangle$ is closed for all $ 0 \leq i < x <  j$ with $j-i > L$}\}, 
\end{equation*}
and, for each $n \geq 1$,
\begin{equation*} 
x_{n+1}^L = \min\{ x > x_n^L \colon \textrm{ $\langle i,j  \rangle$ is closed for all $ x_n^L \leq i < x <  j$ with $j-i > L$}\}.
\end{equation*}
Similarly, for $x_{n}^L$, with $n \leq -1$.

We recover the original sequence $(x_k)_{k\in \mathbb{Z} \setminus \{0\}}$ simply by taking $L=1$.
Writing,
\begin{equation*}
\theta_n^L = x_{n+1}^L - x_{n}^L \quad \textrm{ and } \quad \theta_{-n}^L = x_{-n}^L - x_{-(n+1)}^L.
\end{equation*} 
gives rise to an i.i.d.\ sequence $\{ \theta_k^L \colon k \in \mathbb{Z} \setminus \{0\}\}$. 
 Let ${\theta^L}$ be a random variable distributed as any of the elements in that sequence.
 The content of the next result is similar to that in Lemma~\ref{Lemmaxi} except that $\theta$ is replaced by ${\theta^L}$ for $L$ large.

\begin{lemma}\label{LemmaXiL} 
Let $0 < \delta < s-2$.
There exists $L_o = L_o(s,\delta)$ such that, for every $L \geq L_o(s,\delta)$,
\begin{equation*}
\E\Big[{(\theta^L)}^{1+\delta}\Big] < \infty.
\end{equation*}
\end{lemma}

Before we proceed to prove Lemma \ref{Lemmaxi} let us introduce the \textit{backward exploration}, reminiscent of a procedure found in \cite[Section 2]{GGNR}.
\begin{equation*}
Y^L = \sup \big\{ i < 0 \colon \text{$\langle i,j\rangle$ is open for some $j > 0$ with $j-i > L$} \big\}
\end{equation*}
with the convention $\sup \varnothing = -\infty$.
Let us denote $a(L)=\P(Y^L = -\infty)$.
Since $s > 2$, $a(L)$ is positive and converges to $1$ as $L$ goes to $\infty$, in fact,
\begin{equation}
\label{eq:aL}
\begin{split}
a(L) & = \prod_{i < 0} \prod_{\substack{j > 0 \\ j > i+L } } \P\big(\text{$\langle i,j \rangle$ is closed}\big)\\
& = \prod_{k \geq 2}\prod_{\substack{\ell \geq k \\ \ell > L }} \left(1 - \frac{1}{\ell^s} \right)    \underset{L \to \infty}{\longrightarrow} 1.
\end{split}
\end{equation}
We may omit the dependence on $L$.
Let $Y_1, Y_2, \dots$ be i.i.d.\ random variables with the same distribution as $Y$.
It is readily seen that for every positive integer $n$  the following upper bound holds:
\begin{equation*}
\P \bigg( \sum_{k=1}^m Y_k = -n \bigg) \leq \big(1-a(L)\big)^m,
\end{equation*}
for all $m\geq 1$.
For our purposes, we will need a control that depends on $n$:
\begin{lemma}
\label{l:hit_z}
There exists a constant $c_o>0$ not depending on $L$ such that for every $m \geq 1$ and $n \geq 1$,
\begin{equation}
\label{eq:hit_z}
\P\bigg( \sum_{k=1}^m Y_k = -n  \bigg) <  \frac{c_o^m}{n^{s-1}}.
\end{equation}
\end{lemma}

\begin{proof}
We have
\begin{align*}
\P\left( Y_1 = -n  \right) &= \bigg[\, \prod_{k = 1}^{n-1}\prod_{\substack{\ell > k \\ \ell > L }} \left(1 - \frac{1}{\ell^s} \right) \bigg] \bigg[ 1 - \prod_{\substack{\ell > n \\ \ell > L }} \left(1 - \frac{1}{\ell^s} \right) \bigg] \\
&\leq  \bigg[ 1 - \prod_{\substack{\ell > n \\ \ell > L }} \left(1 - \frac{1}{\ell^s} \right) \bigg] \leq \frac{c_1}{n^{s-1}},
\end{align*}
for some positive constant $c_1 > 0$ that does not depend on either $n$ nor on $L$.
Therefore,
\begin{equation*}
\P\left( Y_1+Y_2 = -n  \right) = \sum_{k=1}^{n-1} \P\left( Y_1 = k  \right) \P\left( Y_2 = n-k  \right) \leq c_1^2 \sum_{k=1}^{n-1} \frac{1}{k^{s-1}}\frac{1}{(n-k)^{s-1}}.
\end{equation*}

For the sum appearing in the right-hand side, ignoring $k=1$ and $k=n-1$ we obtain the upper bound
\begin{align*}
\sum_{k=2}^{n-2} \frac{1}{k^{s-1}} \frac{1}{(n-k)^{s-1}} &= \frac{1}{n^{2s-3}} \sum_{k=2}^{n-2} \frac{1}{(k/n)^{s-1}}\frac{1}{(1 - k/n)^{s-1}}\frac{1}{n}\\ 
&<  \frac{2}{n^{2s-3}} \int_{\frac{1}{n}}^{\frac{1}{2}} \frac{1}{x^{s-1}(1-x)^{s-1}} dx \\
&<  \frac{2}{n^{2s-3}} \int_{\frac{1}{n}}^{\frac{1}{2}} \frac{1}{x^{s-1}(1/2)^{s-1}} dx \\
&< \frac{2^s}{n^{2s-3}}\frac{n^{s-2}}{s-2} = \frac{c_2}{n^{s-1}}.
\end{align*}
Including the terms $k=1$ and $k=n-1$ which are equal to $1/(n-1)^{s-1}$, we get   
\begin{equation*}
\P\left( Y_1+Y_2 = -n  \right) < \frac{c_1^2 c_3}{n^{s-1}},
\end{equation*}
where $c_3 > 0$ is some positive constant that does not depend neither on $n$ nor on $L$.

We can iterate the argument in order to obtain inductively,
\begin{equation*}
\P\bigg( \sum_{k=1}^m Y_k = -n  \bigg) < \frac{c_1^m c_3^{m-1}}{n^{s-1}} <  \frac{c_0^m}{n^{s-1}},
\end{equation*}
for $m \geq 1$, for a sufficiently large $c_0 > 1$ that does not depend neither on $n$ nor on $L$.
\end{proof}

\begin{proof}[Proof of Lemma \ref{LemmaXiL}]
The proof consists in obtaining a suitable estimate for the tail probabilities $\P( \theta^{L}_1 \geq n)$.
Since
\begin{equation}
\P( \theta^{L}_1 \geq n) = \sum_{z \in \mathbb{Z}_{+}} \P( \theta^{L}_1 \geq n, \, x_1^{L}=z).
\end{equation}
we now fix $z \in \mathbb{Z}_{+}$ and concentrate on estimating $\P( \theta^{L}_1 \geq n,\, x_1^{L}=z)$.

On the event $\{\theta_1^L \geq n,~ x^L_1 = z\}$ we have that, for every $z' \in (z, z+n]$, 
\begin{equation*}
\big\{ z \leq i < z' \colon \text{$\langle i,j\rangle$ is open for some $j > z'$ with $j-i > L$} \big\} \neq \varnothing,
\end{equation*}
otherwise we would obtain $x_{2}^L < z' \leq z+n$, hence $\theta_1^L < n$.
Therefore, on the same event, we can define recursively a  sequence by setting $z_1 = z+n$, and for $\ell \geq 1$, we define
\begin{equation*}
z_{\ell + 1} = \max\big\{ z \leq i < z_{\ell} \colon \textrm{$\langle i,j\rangle$ is open for some $j > z_{\ell}$ with  $j-i > L$} \big\},
\end{equation*}
if $z_{\ell} \neq z$, and if $z_{\ell} =z$, we define $z_{\ell +1} = z$.
Notice that this sequence has to stabilize at $z$, which means that there exists a random index $m$, such that $z= z_{m+1} = z_{m+2} = \cdots$.
Moreover, given $z_\ell$, each of the increments $z_{\ell+1} - z_{\ell}$ has the distribution of a  backward exploration conditioned on being larger than $z_{\ell} - z$.
Thus, we can write
\begin{equation*}
\begin{split}
 \P\left(\theta_1^L \geq n ~\vert~ x_1^L = z \right) & \leq \P\left(\exists~ m \geq 1 \textrm{ such that } z_{m+1} = z\right) \\
  & = \P\Big(\exists~ m \geq 1 \colon \sum_{k=1}^m Y_k = -n \Big).
\end{split}
\end{equation*}

Recall that we are assuming $\delta \in (0,s-2)$.
Fix $\alpha = (s-2-\delta)/2$. 
Pick $\beta > 0$ so small that $c_o^{\lceil \beta \log n \rceil} < n^{\alpha}$ for every $n \geq 1$.
We now pick $L_o = L_o(s,\delta)$ (in particular, not depending on $n$) such that for every $L> L_o$, we have $(1-a(L))^{\lceil \beta \log n \rceil} < 1/n^{1+\delta}$, which is possible in view of \eqref{eq:aL}.
Writing $M = \lceil \beta \log n \rceil$, we have
\begin{equation}
\begin{split}
 \P\Big(\exists~ m \geq 1  \colon \sum_{k=1}^m Y_k = -n \Big)
&\leq \sum_{m = 1}^{M} \P\Big( \sum_{k=1}^m Y_k = -n  \Big) + \sum_{m > M} \P\Big( \sum_{k=1}^m Y_k =- n  \Big) \\
&< \sum_{m = 1}^{M}  \frac{c_o^m}{n^{s-1}} + \sum_{m > M} (1-a)^m \\
& < \frac{c_o^{M}}{n^{s-1}} + \frac{ (1-a)^{M} }{a} < \frac{c}{n^{1+\delta}}.
\end{split}
\end{equation}
Then, for every sufficiently large $L$ we have
$\E[(\theta^L_1)^{1+\delta}] < \infty.$
\end{proof}

We now possess sufficient tools to prove Lemma~\ref{Lemmaxi} and Lemma~\ref{Lemmaeta}.

\begin{proof}[Proof of Lemma \ref{Lemmaxi}] Given $\delta \in (0, s-2)$, by Lemma \ref{LemmaXiL} we can fix $L$ such that the random variables $\theta_k^L$ have finite $(1+\delta)$-moment.

Write $A= \{x_k \colon k \in \mathbb{Z} \setminus \{0\}\}$ and $A_L = \{x_k^L \colon k \in \mathbb{Z} \setminus \{0\}\}$. 
Because $A \subset A_L$, for every $k \in \mathbb{Z}_{+}$ and $z\in \mathbb{Z}_+$, we can write
\[
\{x_k = z\} =  \bigcup_{m \geq k}\{x_m^L = x_k = z\},
\]
where the events appearing on the right-hand side are disjoint.

In order to explore independence it will be convenient to look at a subsequence of $A$ with indexes spaced by $L$ units.
Conditional on $\{x_m^L = x_k = z\}$, the events $\{x_{m+jL}^L \in A\}$, ${j \geq 1}$, are independent and occur with probability
\[
\kappa=\P\big(\text{$\langle -i,j \rangle$ is closed for every $i,j > 0$ such that $i+j \leq L$} \big)>0.
\]
Therefore, conditional on $\{x_m^L = x_k = z\}$ the random variable
\begin{equation*}
N = \min\{j \geq 1 \colon x_{m+jL}^L \in A\},
\end{equation*}
is independent of $\{\theta_i^L, i \leq m\}$ and is distributed as a geometric random variable with parameter $\kappa$.
Writing $m' = m+NL$ and $m^{\ast} = \min\{i > m \colon  x_i^L \in A\}$, then $m^{\ast} \leq m'$.
Therefore, defining $\zeta = x_{m'}^L - x_{m}^L$,
on the event $\{x_m^L = x_k = z\}$, we have
\[
\theta_k = x_{m^{\ast}}^L - x_{m}^L \leq x_{m'}^L - x_{m}^L = \zeta.
\] 
Therefore, $\P(\theta_k \geq n ~\vert~ x_m^L = x_k = z) \leq \P(\zeta \geq n ~\vert~ x_m^L = x_k = z)$.

We can write
\begin{equation*}
\zeta =  \sum_{i = m}^{m + LN - 1} \theta_{i}^L.
\end{equation*}

Now, if $\{W_i, ~i \geq 1\}$, are i.i.d.\ random variables with the same distribution as $\theta_i^L$, and $G$ is a geometric random variable with parameter $\kappa$ independent of all the rest, then
the conditional distribution of $\zeta$ given $\{x_m^L = x_k = z\}$ is the same as 
\[
W = \sum_{i=1}^{LG} W_i.
\]
Therefore,
\begin{equation*}
 \P\left(\theta_k \geq n ~\vert~ x_m^L = x_k = z\right) \leq \P\left(\zeta \geq n ~\vert~ x_m^L = x_k = z \right) = \P\left( W \geq n \right).
\end{equation*}
   
Therefore,
\begin{align*}
\P(\theta_k \geq n) &= \P\left(\theta_k \geq n ~\vert~  x_k = z\right) \\
&= \frac{1}{\P(x_k = z)}\sum_{m \geq k} \left[ \P\left(\theta_k \geq n ~\vert~ x_m^L = x_k = z\right)\P\left( x_m^L = x_k = z \right)\right] \\
&\leq \frac{\P(W \geq n)}{\P(x_k = z)}\sum_{m \geq k}\P\left( x_m^L = x_k = z \right) \\ 
&= \P(W \geq n).
\end{align*}

Now, we claim that $\E[W^{1+\delta}] < \infty$. 
In fact, Lemma~\ref{LemmaXiL}, guarantees that $\E[W_i^{1+\delta}] < \infty$ for every $i \geq 1$.
Writing $W = \sum_{\ell \geq 1} \sum_{i = 1}^{\ell L} W_i 1_{\{G=\ell\}}$, and recalling that $G$ has geometric distribution and is independent of the random variables $W_i$, $i \geq 1$, by Minkowski inequality we have for every positive integer $M$,
\begin{equation}\label{Minkowski}
\Big\lVert \sum_{\ell = 1}^M \sum_{i=1}^{\ell L} W_i 1_{{\{G=\ell}\}}\Big\rVert_{1+\delta} \leq \big\lVert W_1\big\rVert_{1+\delta} L \kappa^{1/(1+\delta)} \sum_{\ell=1}^{\infty} \ell  (1-\kappa)^{(\ell-1)/(1+\delta)} < \infty,
\end{equation}
The claim follows by taking the limit as $M$ goes to infinity.
This implies the desired moment bound $\E[\theta^{1+\delta}] \leq \E[W^{1+\delta}] < \infty$.
\end{proof}

\begin{proof}[Proof of Lemma~\ref{Lemmaeta}]
It is sufficient to consider $\eta_1$. By definition, $\eta_1 = v_2 - v_1 = \sum_{j=2i_1}^{2i_2 - 1} \theta_j$. Observe that $i_2 - i_1$ have geometric distribution with parameter $q = [\P(\theta=1)]^2$.

Consider the auxiliary random variable $\tilde{\theta}$ define as
\[\P(\tilde{\theta} = n) =  \P(\theta_1 + \theta_2 = n ~\vert~ \theta_1 + \theta_2 \neq 2), ~n \geq 3.\]
A straightforward consequence of Lemma~\ref{Lemmaxi} is that $\E[\tilde{\theta}^{1+\delta}] < \infty.$  

Let $\{\tilde{\theta}_j, ~j \geq 1\}$ be a sequence of i.i.d.\ copies of $\tilde{\theta}$ and $G$ a geometric random variable with parameter $q$ independent of this sequence. In this way, we have the following equality in distribution
\begin{equation*}
\begin{split}
\eta_1 &= \sum_{j=2i_1}^{2i_2 - 1} \theta_j = \theta_{i_1} + \theta_{2i_2 - 1} + \left[\sum_{j = i_1+1}^{i_2 - 1} (\theta_{2j-1} + \theta_{2j})\right]1_{\{i_2 \neq i_1 +1\}}~ \\ &\overset{D}{=} ~   2 + \left[\sum_{j = 1}^{G - 1}  \tilde{\theta}_j \right] 1_{\{G \neq 1\}}.
\end{split}
\end{equation*}

The conclusion of the proof follows arguing with a Minkowski inequality application as done in the paragraph of \eqref{Minkowski}. 
\end{proof}  

We now move to the proof of Lemma \ref{l:expectation_ek}.
As a first step, we have the following result about the expected number of open edges inside a deterministic interval. 
For each positive integer $M$, let
\[
N_M = \big\lvert \{\langle x,y \rangle  \colon \min\{x,y\} \in \llbracket 1,M \rrbracket, \langle x,y \rangle \text{ is open}\} \big\rvert,
\]
where the notation $[[a,b]]$ denotes the set $[a,b]\cap\Z$.

\begin{lemma}\label{LemmaCaixa}
Given $\alpha > 0$, for every $M$ sufficiently large,
\begin{equation*}
\P\big( N_M \geq 2M^{1+\alpha} \big) \leq \frac{1}{M^{1+(3\alpha/2)}}.
\end{equation*}
\end{lemma}
\begin{proof}
It is straightforward to verify that
$M \leq \E[N_M] = M \E[N_1]$ and $\mathbb{V}\textrm{ar}[N_M] = M \E[N_1]$.
Hence, for $M$ such that $\mathbb{E}[N_M] < M^{1+\alpha}$,
\begin{align*}
 \P(N_M \geq 2M^{1+\alpha})
&\leq \P\left((N_M - \E[N_M])^2 \geq M^{2+2\alpha}  \right) \\
&\leq \frac{\mathbb{V}\textrm{ar}[N_M]}{M^{2+2\alpha}} \leq \frac{\mathbb{E}[N_1]}{M^{1+2\alpha}}.
\end{align*}
Hence, the assertion holds as soon as $\mathbb{E}[N_1] < M^{\alpha/2}$.
\end{proof}

We are now ready to prove Lemma \ref{l:expectation_ek}.

\begin{proof}[Proof of Lemma \ref{l:expectation_ek}]
It suffices to consider even values of $k$. Without loss of generality, we also assume that $k$ is positive.
Let $\delta \in (0,s-2)$, and define $\varepsilon = (1+\delta)/(1+\alpha)-1$, where $\alpha = 5\delta/6$.

For every $n \geq 1$, we have
\begin{equation}
\label{e:tail_ek}
\P  \Big(  |\mathcal{E}_k| \geq n^{1/(1+\varepsilon)}\Big) 
\leq \P\Big( 2|V_k| \geq n^{1/(1+\delta)}  \Big) + \P\Big(  |\mathcal{E}_k| \geq n^{1/(1+\varepsilon)}, 2|V_k| < n^{1/(1+\delta)}  \Big).
\end{equation}

For a vertex $z \in \Z$ and a positive integer $M \geq 1$, let 
\[
N_M^z = \big\vert \big\{\langle x,y \rangle \in \mathcal{E}  \colon \min\{x,y\} \in \llbracket z+1,z+M \rrbracket \big\} \big\vert,
\]
be the number of open edges linking a vertex in $\llbracket z+1,z+M \rrbracket$ to another vertex lying to its right.
Conditioning on $\{ v_{k/2} = z\}$ and using  Lemma~\ref{LemmaCaixa} we get, for $M$ sufficiently large (depending on $\delta$).
\begin{equation}
\begin{split}
\P\left( |\mathcal{E}_k| \geq 2M^{1+\alpha}, |V_k| < M \mid v_{k/2} = z \right) 
& \leq \P \left( N_{M}^z \geq 2M^{1+\alpha} \mid v_{k/2} = z \right) \\
& \leq \frac{1}{M^{1+3 \alpha/2}}.
\end{split}
\end{equation}
Summing over $z \in \mathbb{Z}$ we obtain 
\[
\P\left(|\mathcal{E}_k| \geq 2M^{1+ \alpha}, |V_k| < M \right) \leq 1/M^{1+3\alpha/2}.
\]
Taking $M = n^{1/(1+\delta)}$
\begin{equation*}
\P\left(  |\mathcal{E}_k| \geq  2 n^{1/(1+\varepsilon)}, |V_k| < n^{1/(1+\delta)}  \right) \leq \frac{1}{n^{\beta}},
\end{equation*}
where $\beta = [1+(3\alpha/2)]/(1+\delta) > 1$.

Recalling \eqref{moment}, and summing \eqref{e:tail_ek} over all $n \geq 1$, we obtain
\begin{align*} 
\frac{1}{2^{1+\varepsilon}}\E[|\mathcal{E}_k|^{1+\varepsilon}]  &= \sum_{n \geq 1} \P\left( |\mathcal{E}_k| \geq 2 n^{1/(1+\varepsilon)}\right) \\ 
&\leq  \sum_{n \geq 1} \P\left( |V_k| \geq n^{1/(1+\delta)} \right) + \sum_{n \geq 1} \P\left(  |\mathcal{E}_k| \geq  2 n^{1/(1+\varepsilon)}, |V_k| < n^{1/(1+\delta)}  \right) \\
&\leq \E[|V_k|^{1+\delta}] + \sum_{n \geq 1} \frac{1}{n^{\beta}} < \infty.
\end{align*}
In conclusion, $\E[|\mathcal{E}_k|^{1+\varepsilon}] < \infty$ as desired.
\end{proof}

\section{The renormalization step}
\label{s:renormalization}

We now implement a one-step renormalization to couple the contact process on \( G = (\mathbb{Z}, \mathcal{E}) \) with a site percolation model on \( \mathbb{Z} \times \mathbb{Z}_+ \).  

We begin by partitioning \( \mathbb{R}_+ \) into intervals of length \( T = 1/{\sqrt{\lambda}} > 0 \).
Given \( \mathcal{E} \), a pair \( (k,j) \in \mathbb{Z} \times \mathbb{Z}_+ \) is said to be good if, during the \( j \)-th time interval
\begin{itemize}
\item there are no infection attempts along any edge \( e \in \mathcal{E}_k \):
\begin{equation*}
\textrm{$\mathcal{I}_e \cap [jT, (j+1)T) = \varnothing$,\, for every $e\in \mathcal{E}_k$},
\end{equation*}
\item for each vertex \( x \in V_k \), there is at least one recovery mark: 
\begin{equation*}
\textrm{$\mathcal{R}_x \cap [jT, (j+1)T) \neq \varnothing$, \, for every vertex $x \in V_k$}.
\end{equation*}
\end{itemize}
We define 
\[
p = (1-  e^{-T})e^{-1/T}
\]
whose value can be made arbitrarily close to $1$ by decreasing the infection rate $\lambda$ (hence increasing $T$).
Under the quenched probability measure governing the contact process on $G=(V,\mathcal{E})$, we have for every realization of $\mathcal{E}$,
\begin{equation}
\label{e:domination_good}
P^{\mathcal{E}}_{\lambda} \left( (k,j) \textrm{ is good}\right) = \left(1 - e^{-T}\right)^{|V_k|}e^{-\lambda T|\mathcal{E}_k|} \geq p^{|\mathcal{E}_k|},
\end{equation}
where we have used the fact that $|V_k| \leq |\mathcal{E}_k|$.

Let $\Phi = (\phi_k)_{k \in \Z}$ be a sequence of independent positive random variables that we call a random environment.
Conditional on $\Phi$, we denote by $P_{\Phi, p}$ the law of a site percolation process on $\mathbb{Z} \times \mathbb{Z}_+$ under which every site $(i,j)$ is open independently with probability $p^{\phi_i}$.
In what follows, we will be interested in the setting where $\phi_k = |\mathcal{E}_k|$ for every $k$, because by \eqref{e:domination_good} $P_{\Phi,p}$ is dominated stochastically by the law of the random process $(\mathbf{1}_{\text{$(k,j)$ is good}})_{(k,j)\in \mathbb{Z}\times \mathbb{Z}_+}$.

Since $(v_k)_{k \in \mathbb{Z} \setminus \{0\}}$ are cut-points  there are no open edges connecting $V_{k}$ to $V_{\ell}$ if $|k - \ell|>1$.
Consequently, if the infection enters a good box, it can only propagate to an adjacent box sharing a common face. 
However, this is impossible if the latter is also good.
This implies that the existence of a semi-circuit of good vertices around the origin in $\mathbb{Z} \times \mathbb{Z}_+$, ensuring that the contact process dies out.

\begin{center}
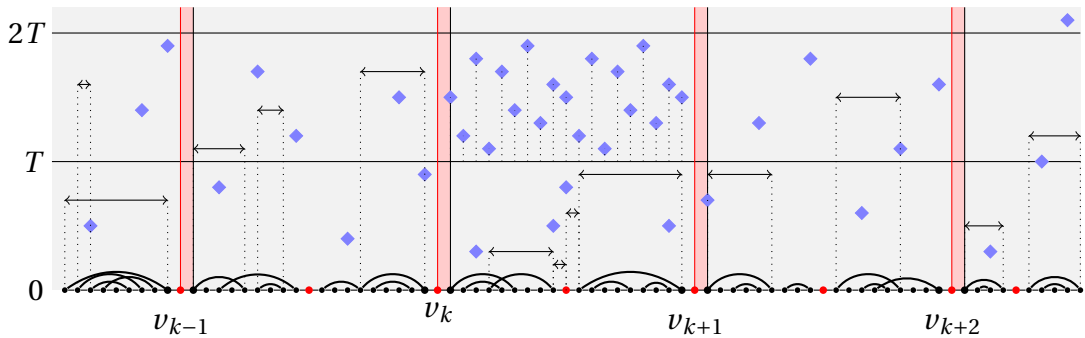
\begin{figure}[htb!]
\begin{tikzpicture}[scale=0.17, every node/.style={circle, fill, inner sep=.75pt}]
    \def\n{80}
    
    \foreach \x in {10, 30, 50, 70} {
        \fill[red!20] (\x,0) rectangle (\x+1,10); 
        \fill[red!20] (\x,10) rectangle (\x+1,20); 
        \fill[red!20] (\x,20) rectangle (\x+1,22); 
    }

    \fill[gray!10] (0,0) rectangle (9,10);    
    \fill[gray!10] (0,10) rectangle (9,20);
    \fill[gray!10] (0,20) rectangle (9,22);
    \fill[gray!10] (9,0) rectangle (10,10);   
    \fill[gray!10] (9,10) rectangle (10,20);
    \fill[gray!10] (9,20) rectangle (10,22);
    \fill[gray!10] (11,0) rectangle (29,10);  
    \fill[gray!10] (11,10) rectangle (29,20);
    \fill[gray!10] (11,20) rectangle (29,22);
    \fill[gray!10] (29,0) rectangle (30,10);  
    \fill[gray!10] (29,10) rectangle (30,20);
    \fill[gray!10] (29,20) rectangle (30,22);
    \fill[gray!10] (31,0) rectangle (49,10);  
    \fill[gray!10] (31,10) rectangle (49,20);
    \fill[gray!10] (31,20) rectangle (49,22);
    \fill[gray!10] (49,0) rectangle (50,10);  
    \fill[gray!10] (49,10) rectangle (50,20);
    \fill[gray!10] (49,20) rectangle (50,22);
    \fill[gray!10] (51,0) rectangle (69,10);  
    \fill[gray!10] (51,10) rectangle (69,20);
    \fill[gray!10] (51,20) rectangle (69,22);
    \fill[gray!10] (69,0) rectangle (70,10);  
    \fill[gray!10] (69,10) rectangle (70,20);
    \fill[gray!10] (69,20) rectangle (70,22);
    \fill[gray!10] (71,0) rectangle (80,10);  
    \fill[gray!10] (71,10) rectangle (80,20);
    \fill[gray!10] (71,20) rectangle (80,22);

    \foreach \x in {1,...,\n} {
        \node (s\x) at (\x, 0) {};
    }
    \foreach \x in {1,...,79} {
        \draw (s\x) -- (s\the\numexpr\x+1);
    }

    \def\angleShort{40}  
    \def\angleMedium{50} 
    \def\angleLong{35}   

    \draw[thick] (s1) to[out=\angleLong, in=180-\angleLong] (s9);
    \draw[thick] (s2) to[out=\angleMedium, in=180-\angleMedium] (s7);
    \draw[thick] (s3) to[out=\angleShort, in=180-\angleShort] (s6);
    \draw[thick] (s4) to[out=\angleMedium, in=180-\angleMedium] (s8);
    \draw[thick] (s11) to[out=\angleMedium, in=180-\angleMedium] (s15);
    \draw[thick] (s13) to[out=\angleShort, in=180-\angleShort] (s19);
    \draw[thick] (s16) to[out=\angleShort, in=180-\angleShort] (s18);
    \draw[thick] (s21) to[out=\angleShort, in=180-\angleShort] (s24);
    \draw[thick] (s24) to[out=\angleMedium, in=180-\angleMedium] (s29);
    \draw[thick] (s25) to[out=\angleShort, in=180-\angleShort] (s28);
    \draw[thick] (s31) to[out=\angleMedium, in=180-\angleMedium] (s36);
    \draw[thick] (s32) to[out=\angleShort, in=180-\angleShort] (s35);
    \draw[thick] (s34) to[out=\angleMedium, in=180-\angleMedium] (s39);
    \draw[thick] (s41) to[out=\angleLong, in=180-\angleLong] (s49);
    \draw[thick] (s42) to[out=\angleShort, in=180-\angleShort] (s45);
    \draw[thick] (s46) to[out=\angleMedium, in=180-\angleMedium] (s48);
    \draw[thick] (s51) to[out=\angleMedium, in=180-\angleMedium] (s56);
    \draw[thick] (s53) to[out=\angleShort, in=180-\angleShort] (s55);
    \draw[thick] (s57) to[out=\angleShort, in=180-\angleShort] (s59);
    \draw[thick] (s61) to[out=\angleMedium, in=180-\angleMedium] (s66);
    \draw[thick] (s63) to[out=\angleShort, in=180-\angleShort] (s65);
    \draw[thick] (s64) to[out=\angleLong, in=180-\angleLong] (s69);
    \draw[thick] (s71) to[out=\angleMedium, in=180-\angleMedium] (s74);
    \draw[thick] (s72) to[out=\angleShort, in=180-\angleShort] (s73);
    \draw[thick] (s76) to[out=\angleMedium, in=180-\angleMedium] (s80);
    \draw[thick] (s77) to[out=\angleShort, in=180-\angleShort] (s79);

    \foreach \x in {10, 20, 30, 40, 50, 60, 70, 75} {
        \node[red] at (\x,0) [circle, fill=red, inner sep=1pt] {};
    }
    \node[below, draw=none, fill=none] at (10,0) {$v_{k-1}$};
    \node[below, draw=none, fill=none] at (30,0) {$v_{k}$};
    \node[below, draw=none, fill=none] at (50,0) {$v_{k+1}$};
    \node[below, draw=none, fill=none] at (70,0) {$v_{k+2}$};

    \draw[thin] (0,10)--(80,10);
    \draw[thin] (0,20)--(80,20);

    \foreach \x in {10, 30, 50, 70} {
        \draw[thin, red] (\x,0) -- (\x,22);
        \draw[thin] (\x+1,0) -- (\x+1,22);
        \node at (\x-1,0) [circle, fill=black, inner sep=1pt] {};
        \node at (\x+1,0) [circle, fill=black, inner sep=1pt] {};
    }

    \node[left, draw=none, fill=none] at (0,0) {$0$};
    \node[left, draw=none, fill=none] at (0,10) {$T$};
    \node[left, draw=none, fill=none] at (0,20) {$2T$};

    \foreach \h/\th in {31/15, 32/12, 33/18, 34/11, 35/17, 36/14, 37/19, 38/13, 39/16,
                        40/15, 41/12, 42/18, 43/11, 44/17, 45/14, 46/19, 47/13, 48/16, 49/15} {
        \draw[blue!50, fill] (\h,\th-0.5) -- (\h+0.5,\th) -- (\h,\th+0.5) -- (\h-0.5,\th) -- cycle;
        \draw[dotted, thin] (\h,10)--(\h,\th);
    }
    \foreach \h/\th in {33/3, 39/5, 40/8,
                          48/5} {
        \draw[blue!50, fill] (\h,\th-0.5) -- (\h+0.5,\th) -- (\h,\th+0.5) -- (\h-0.5,\th) -- cycle;
    }

    \draw[thin, <->] (40,6) -- (41,6);
    \draw[thin, <->] (40,2) -- (39,2);
    \draw[thin, <->] (39,3) -- (34,3); 
    \draw[thin, <->] (41,9) -- (49,9);
    \draw[thin, dotted] (40,6) -- (40,0); 
    \draw[thin, dotted] (41,6) -- (41,0);
    \draw[thin, dotted](40,2) -- (40,0);
    \draw[thin, dotted](39,2)  -- (39,0);
    \draw[thin, dotted] (39,3) -- (39,0); 
    \draw[thin, dotted] (34,3) -- (34,0);
    \draw[thin, dotted] (41,9) --(41,0);
    \draw[thin, dotted](49,9) -- (49,0);

    \foreach \h/\th in {3/5, 7/14, 9/19, 13/8, 16/17, 19/12, 23/4, 27/15, 29/9,
                        51/7, 55/13, 59/18, 63/6, 66/11, 69/16, 73/3, 77/10, 79/21} {
        \draw[blue!50, fill] (\h,\th-0.5) -- (\h+0.5,\th) -- (\h,\th+0.5) -- (\h-0.5,\th) -- cycle;
    }

    \draw[thin, <->] (1,7) -- (9,7);    
    \draw[thin, <->] (2,16) -- (3,16);  
    \draw[thin, <->] (11,11) -- (15,11); 
    \draw[thin, <->] (16,14) -- (18,14); 
    \draw[thin, <->] (24,17) -- (29,17); 
    \draw[thin, <->] (51,9) -- (56,9);  
    \draw[thin, <->] (61,15) -- (66,15); 
    \draw[thin, <->] (71,5) -- (74,5);  
    \draw[thin, <->] (76,12) -- (80,12); 

    \draw[thin, dotted] (1,7) -- (1,0) (9,7) -- (9,0);    
    \draw[thin, dotted] (2,16) -- (2,0) (3,16) -- (3,0);  
    \draw[thin, dotted] (11,11) -- (11,0) (15,11) -- (15,0); 
    \draw[thin, dotted] (16,14) -- (16,0) (18,14) -- (18,0); 
    \draw[thin, dotted] (24,17) -- (24,0) (29,17) -- (29,0); 
    \draw[thin, dotted] (51,9) -- (51,0) (56,9) -- (56,0);  
    \draw[thin, dotted] (61,15) -- (61,0) (66,15) -- (66,0); 
    \draw[thin, dotted] (71,5) -- (71,0) (74,5) -- (74,0);  
    \draw[thin, dotted] (76,12) -- (76,0) (80,12) -- (80,0); 
\end{tikzpicture}
\caption{The renormalization scheme. 
We represent the event $\{(2k,1) \text{ is good}\}$.
In the corresponding rectangle $\llbracket v_{k}+1, v_{k+1}-1 \rrbracket \times [T,2T]$ every site has a recovery mark (lozenge) and no edge has an infection attempt (double arrow).}
\label{f:renormalization_boxes}
\end{figure}
\end{center}

\section{Crossing events}
\label{sec:crossing}

In this section, we relate the crossing probabilities of  rectangles under $P_{\Phi,p}$ and under the bond percolation measure $\mathbb{P}^{\Lambda}_\rho$ defined on \eqref{e:def_perc_stret}, for suitable choices of $\Lambda$ and $\rho$ depending on $\Phi$ and $p$.

We will fix $\Phi = (\phi_k)_{k\in\mathbb{Z}}$ where $\phi_k = |\mathcal{E}_k|$ for every $k$.

A subset $R = \llbracket a,b \rrbracket \times \llbracket c,d \rrbracket \subset \Z \times \Z_+$ where $ a < b $ and $0 \leq c < d $ are integer numbers is called a rectangle.
We will be interested in the occurrence of horizontal and vertical crossings spanning $R$:
\begin{align*}
\mathcal{C}_{h}(R) &= \big\{ \{a\} \times \llbracket c,d \rrbracket \overset{R}\longleftrightarrow  \{b\} \times \llbracket c,d \rrbracket \big\}, \\ 
\mathcal{C}_{v}(R) &= \big\{  \llbracket a,b \rrbracket \times \{c\} \overset{R}\longleftrightarrow   \llbracket a,b \rrbracket \times \{d\} \big\}
\end{align*}
where $A \overset{S}\longleftrightarrow B$ stands for the existence of a path of open vertices connecting $A$ to $B$ while remaining inside $S$.
We abuse notation and still write $\mathcal{C}_{h}(R)$ and $\mathcal{C}_{v}(R)$ for crossings of open edges in the rectangle $R$.

The next result compares the crossing probabilities of a rectangle $R$ and a rescaled version of it under the site percolation measure $P_{\Phi,p}$ introduced in the previous section and the bond percolation measure $\mathbb{P}^{\Lambda}_\rho$, respectively. 
The content of the following lemma is based on a standard coupling of site and bond percolation and is reminiscent of \cite[Theorem 1.33]{Grimm}.

\begin{lemma}
\label{l:compare_site_bond}
Let $\Lambda = \{\ldots, \phi_{-2}+1, \phi_0+1, \phi_2+1, \phi_4+1, \ldots \}$ and $\rho = 1 - (1-p)^{-1/4}$.
Then
\begin{eqnarray}
P_{\Phi,p}\big(\mathcal{C}_{h}( \llbracket 2a+1,2b+1 \rrbracket \times\llbracket c,d \rrbracket )\big)  
& \geq & \P_{\rho}^{\Lambda}\big(\mathcal{C}_{h}(\llbracket a,b \rrbracket \times\llbracket c,d \rrbracket)\big), \\
P_{\Phi,p}\big(\mathcal{C}_{v}( \llbracket 2a+1,2b+1 \rrbracket \times\llbracket c,d \rrbracket )\big), 
& \geq & \P_{\rho}^{\Lambda}\big(\mathcal{C}_{v}(\llbracket a,b \rrbracket \times\llbracket c,d \rrbracket)\big).
\end{eqnarray}

\end{lemma}

\begin{proof}
Let $R = \llbracket 2a+1,2b+1 \rrbracket \times\llbracket c,d \rrbracket $.
Write $e_1 = \langle (0,0), (2,0)\rangle $ and $e_2 = \langle (0,0), (0,1) \rangle$ and define the following set of unoriented edges
\begin{equation*}
e(R) = \big\{ \big\langle (2i+1,j), (2i+1,j) + e_\ell \big\rangle \colon i \in \llbracket a, b\rrbracket, j \in \llbracket c,d \rrbracket, \ell = 1,2 \big\}.
\end{equation*}

Now let $\rho$ be such that $1-p = (1-\rho)^4$ and split each site of the form $(2i+1,j) \in R$  into four sites $(2i+1,j)_1, \dots, (2i+1,j)_4$.
Declare each site $(2i+1,j)_{\ell}$, $\ell = 1, \dots, 4$, open with probability $\rho$, independently.
We can recover the percolation process $P_{\Phi,p}$ inside $R$ by saying that $(2i+1,j)$ is open if, and only if, at least one of its four sites are open.

We now use an exploration process to compare $P_{\Phi,p}(\mathcal{C}_{h}(R))$ with the probability of the same event under the law of the bond percolation on a randomly stretched lattice \eqref{e:def_perc_stret}.

Fix any ordering of the edges in $e(R)$.
Start with 
\begin{equation*}
(A_0,B_0) = \big( \{2a+1\} \times \llbracket c,d \rrbracket, \varnothing \big).
\end{equation*}
Given $(A_n,B_n), n \geq 0$, let
\begin{equation*}
E_n = \{ e \in B_n^c \cap e(R) \colon e = \langle x,y \rangle \textrm{ for some } x \in A_n \textrm{ and } y \notin A_n\}.
\end{equation*}
Inductively, we define  $(A_{n+1},B_{n+1}) = (A_n,B_n)$ if $E_n$ is empty; otherwise, let $e_n = \langle x,y \rangle$ be the earliest edge in $E_n$ with $x = (2i+1, j) \in A_n$.
Set $B_{n+1} = B_n \cup \{e_n\}$, and define $A_{n+1}$ as follows:
\begin{itemize}
\item If $y = x+(2,0)$, set 
\begin{equation}
\label{e:add_y_1}
A_{n+1} =
\begin{cases}
 A_n \cup \{y\}, \textrm{ if $(2i+2, j)$ and $(2i+3,j)_1$ are open,} \\
A_n, \textrm{ otherwise.}
\end{cases}
\end{equation}
\item If $y = x-(2,0)$, set 
\begin{equation}
\label{e:add_y_2}
A_{n+1} =
\begin{cases}
 A_n \cup \{y\}, \textrm{ if $(2i, j)$ and $(2i-1,j)_2$ are open,} \\
A_n, \textrm{ otherwise.}
\end{cases}
\end{equation}
\item If $y = x+(0,1)$, set
\begin{equation}
\label{e:add_y_3}
A_{n+1} =
\begin{cases}
 A_n \cup \{y\}, \textrm{ if $(2i+1, j+1)_3$ is open;,} \\
A_n, \textrm{ otherwise.}
\end{cases}
\end{equation}
\item If $y = x-(0,1)$, set
\begin{equation}
\label{e:add_y_4}
A_{n+1} =
\begin{cases}
 A_n \cup \{y\}, \textrm{ if $(2i+1, j-1)_4$ is open,} \\
A_n, \textrm{ otherwise.}
\end{cases}
\end{equation}
\end{itemize}
At each step of the exploration, the probability of adding vertex $y$ equals either $\rho p^{\phi_{2i}}$, $\rho p^{\phi_{2(i-1)}}$ or $\rho$ depending on whether we are respectively in \eqref{e:add_y_1}, \eqref{e:add_y_2} or in  \eqref{e:add_y_3}, \eqref{e:add_y_4}.

Let $(A_{\infty}, B_{\infty}) = \cup_{n\geq 0}(A_n, B_n)$.
Notice that $A_\infty$ equals the exploration cluster of $A_0$ obtained in the bond percolation model on $e(R)$ where vertical edges are open with probability $\rho$ and horizontal edges $\langle (2i+1,j),(2i+1,j)+(2,0) \rangle$ are open with probability $\rho p^{\phi_{2i}} \geq \rho^{\phi_{2i} + 1}$, $a \leq i < b$.

Therefore,
\begin{align*}
P_{\Phi,p}\big(\mathcal{C}_{h}(R)\big) & 
\geq 
\P\left( A_{\infty}\cap \big(  \{2b+1\} \times \llbracket c,d \rrbracket \big) \neq \varnothing \right) \\ 
& = \P_{\rho}^{\Lambda}\big(\mathcal{C}_{h}(\llbracket a,b \rrbracket \times\llbracket c,d \rrbracket)\big),
\end{align*}
where, we recall, $\Lambda = \{\ldots, \phi_{-2}+1, \phi_0+1, \phi_2+1, \phi_4+1, \ldots \}$ and $\mathbb{P}^{\Lambda}_{\rho}$ is the quenched law governing the percolation process defined in \eqref{e:def_perc_stret}.
In the right-hand side we abuse notation and write $\mathcal{C}_h \big( (\llbracket a,b \rrbracket \times\llbracket c,d \rrbracket) \big)$ for the event that the respective rectangle is crossed horizontally by a path of  open bonds.
The rectangle $\llbracket a,b \rrbracket \times\llbracket c,d \rrbracket$ appears instead of $R$, because we are using edges of length $2$ in our exploration process.

Similarly
\begin{equation*}
P_{\Phi,p}\big(\mathcal{C}_{v}(R)\big) \geq \P_{\rho}^{\Lambda}\big(\mathcal{C}_{v}(\llbracket a,b \rrbracket \times\llbracket c,d \rrbracket)\big).
\end{equation*}
\end{proof}

In next section, we will define specific classes of rectangles such that horizontal and vertical crossing events almost surely occur provided that $p$, and consequently $\rho$, are sufficiently close to $1$.

\section{Percolation on stretched lattices}
\label{sec:stretched}

Let $\xi$ be a positive random variable taking integer values and $\{\xi_k\}_{k\in\mathbb{Z}}$ a sequence consisting of i.i.d.\ copies of $\xi$.
Let $x_0 =0$ and define recursively, 
\begin{equation}
\label{e:def_renewal}
\begin{cases}
x_{k} = x_{k-1}+\xi_{k}, ~ k \in \mathbb{Z}_{+},\\
x_{k} = x_{k+1}-\xi_{k+1}, ~ k \in \mathbb{Z}_{-},
\end{cases}
\end{equation}
and denote
\begin{equation}
\label{e:def_lambda}
\Lambda:= \{\ldots, x_{-2}, x_{-1}, x_0, x_1, x_2,\ldots\} \subseteq \mathbb{Z}.
\end{equation}
Hence $\Lambda$ is a two-sided renewal process with  interarrival distribution $\xi$.
We denote $\upsilon^{0}_\xi(\cdot)$ the law of $\Lambda$.

Conditional on \(\Lambda\), define the graph \(\mathcal{L}_\Lambda = (V(\mathcal{L}_\Lambda), E(\mathcal{L}_\Lambda))\) with vertex set and edge set given respectively by  
\begin{eqnarray}
V(\mathcal{L}_\Lambda) &:=& \Lambda \times \mathbb{Z}_+  \nonumber\\
E(\mathcal{L}_\Lambda) &:=& \big\{\langle(x_i, n), (x_j, m)\rangle \colon |i - j| + |n - m| = 1 \big\}.\nonumber
\end{eqnarray}  
Informally, \(\mathcal{L}_\Lambda\) is a deformation of \(\mathbb{Z} \times \mathbb{Z}_+\) where horizontal edges are stretched so that all edges between the \(k\)-th and \((k+1)\)-th column have random length \(\xi_k=x_{k} - x_{k-1}\).  
See Figure \ref{fig3_1}.

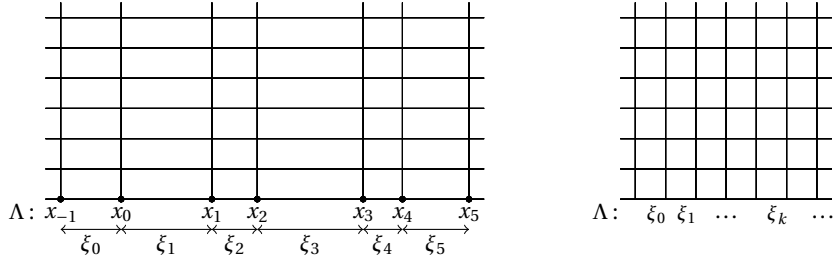
\begin{figure}[htb!]
	\centering
	\begin{tikzpicture}[scale=0.4, every node/.style={scale=0.75}]
	\foreach \y in {0,1,2,3,4,5,6} {
		\foreach\x in {0,2,5,6.5,10,11.3,13.5}	{
			\draw[ ] (\x,0) to (\x, 6.5);
			\draw[ ] (-.5,\y) to (14, \y);
			\fill (\x,0) circle(3pt);
	}}
	\def\numbers{{0,2,5,6.5,10,11.3,13.5}}
	   \foreach \i in {0,...,5}{
	    \pgfmathsetmacro{\n}{\numbers[\i]}
	     \pgfmathsetmacro{\m}{\numbers[\i+1]}
	     \draw[<->] (\n,-1) -- (\m,-1);
	     \node at (\n*.5+\m*.5,-1.5) {$\xi_{\i}$};

	     }
	\node at (0,-0.5) {$x_{-1}$};
	\node at (2,-0.5) {$x_0$};
	\node at (5,-0.5) {$x_1$};
	\node at (6.5,-0.5) {$x_2$};
	\node at (10,-0.5) {$x_3$};
	\node at (11.3,-0.5) {$x_4$};
	\node at (13.5,-0.5) {$x_5$};
	\node at (-1.3,-0.4) {$\Lambda:$};

	\foreach \y in {0,1,2,3,4,5,6} {
		\foreach\x in {0,1,2,3,4,5,6}	{
			\draw[ ] (\x+19,0) to (\x+19, 6.5);
			\draw[ ] (-.5+19,\y) to (6.5+19, \y);
	}}
	\node at (19.7,-0.5) {\small{${\xi_0}$}};
	\node at (20.7,-0.5) {\small{${\xi_1}$}};
	\node at (22,-0.6) {$\ldots$};
	\node at (23.7,-0.5) {\small{${\xi_k}$}};
	\node at (25.2,-0.6) {$\ldots$};
	\node at (18,-0.4) {$\Lambda:$};
	\end{tikzpicture}
	\caption{The lattice $\mathcal{L}_\Lambda$ (on the left) and $\mathbb{Z}\times \mathbb{Z}_+$ (on the right).
	The environment $\Lambda$ can be specified either by the $x_k$'s (left) or by the $\xi_k=x_{k}-x_{k-1}$ (right).
}\label{fig3_1}
\end{figure}

For a given realization of $\Lambda$ and a parameter $p \in [0,1]$, we consider the bond percolation process on $\mathcal{L}_\Lambda$ in which each edge is open independently with probability
\begin{equation}
\label{eq:_p_e}
p_e = p^{\abs{e}},
\end{equation}
where $\abs{e} = \|v_1 - v_2\|$ is the length of the edge $e=\{v_1,v_2\}$.
The point in defining that percolation on the stretched lattice is that it is equivalent to the bond percolation model in $\mathbb{Z}\times\mathbb{Z}_+$ given by \eqref{e:def_perc_stret}.
For this reason we will abuse notation and denote $\mathbb{P}^{\Lambda}_p$ the law of both processes.

We assume that $\xi$ is an integer-valued random variable and that 
\begin{itemize}
\item For $\varepsilon$ given at Lemma \ref{l:expectation_ek}
\begin{equation}
\label{e:xi_na_um_mais_epsilion}
\mathbb{E}(\xi^{1+\varepsilon}) < \infty.
\end{equation}

\item 
 $\xi$ is {\it aperiodic}, that is,
\begin{equation}
\label{e:xi_aperiodic}
\gcd\big\{k\in\mathbb{Z}_+ \setminus \{0\} \colon \mathbb{P}(\xi=k)>0\big\}=1.
\end{equation}
\end{itemize}

\remark{We are mainly interested in the case where $\xi_k = \phi_{2k}+1$, where $\phi_{2k} = |\mathcal{E}_{2k}|$, because
by Lemma \ref{l:compare_site_bond}, in order to bound below crossing probabilities of rectangles under $P_{\Phi,p}$ one may control the equivalent crossing probabilities for $\mathbb{P}^{\Lambda}_\rho$. 
A minor issue is that this sequence is not i.i.d.\ because the law of $|\mathcal{E}_0|$ is different from the others. 
}

\remark{By the independence of the states of each edge in the long-range percolation graph $G= (V, \mathcal{E})$ the $|\mathcal{E}_{2k}|$ may assume any large enough integer value, so the aperiodicity assumption will not pose any serious restriction for our arguments.}

Let us now fix two sequences $L_0, L_1, L_2, \ldots$ and $H_0, H_1, H_2, \ldots$ that will represent the horizontal length and vertical height scales along which we will analyze the model.
They are defined recursively by fixing integers $L_0 > 1$, $H_0 >1$ and constants $\gamma >1$ and $\mu \in (0,1)$ and then setting for every $k\geq1$,
\begin{eqnarray}
\label{Lk}  
L_k=L_{k-1}\lfloor L_{k-1}^{\gamma-1}\rfloor, \qquad \textrm{ and } \qquad  H_k=2\lceil \exp(L_k^\mu)\rceil H_{k-1}.
\end{eqnarray}

\begin{remark} 
For technical reasons, the initial scale $L_0$, has to be chosen sufficiently large, depending on $\varepsilon$.
The reader is referred to \cite[Conditions (i), (ii) and (iii), pp.\ 3153]{Hilario2023} for the exact requirements on $L_0$.
In addition, the exponent $\gamma$ cannot be too large while $\mu$ has to be taken sufficiently close to one.
The following choices
\begin{equation}
\label{e:def_gamma_mu}
 \gamma = 1+\frac{\varepsilon}{2(\varepsilon +4)} \,\,\,\,\,\,\,\,\,\,\,\,\,\,\,\, \textrm{ and }\,\,\,\,\,\,\,\,\,\,\,\,\,\,\,\,\, \mu = \frac{1}{2} \Big(1+\frac{1}{\gamma}\Big).
\end{equation}
will satisfy the requirements
(see \cite[Eq.\ (19) pp.\ 3153 and Eq.(32) pp.\ 3156]{Hilario2023}).
\end{remark}

For $k\in\mathbb{Z}_+$, let us partition $\mathbb{R}_+$ into intervals of length $L_k$
\[
I_j^k:=\big[j L_k, (j+1)L_k\big), \textrm{ with } j\in\mathbb{Z}.
\]
We will be interested in studying the occurrence of horizontal and vertical crossings of rectangles with length of order $L_k$ and height of order $H_k$.

For $i, j, k\in\mathbb{Z}_+$ denote
\begin{eqnarray}
\label{defC}
C_{i, j}^k:=\mathcal{C}_h\Big(\big(I^k_i\cup I^k_{i+1}\big)\times\big[j H_k, (j+1)H_k\big)\Big)
\end{eqnarray}
\begin{eqnarray}
\label{defD}
D_{i, j}^k:=\mathcal{C}_v\Big(I^k_i\times\big[j H_k, (j+2)H_k\big)\Big).
\end{eqnarray}
These events are illustrated in Figure \ref{fig:CeD}.

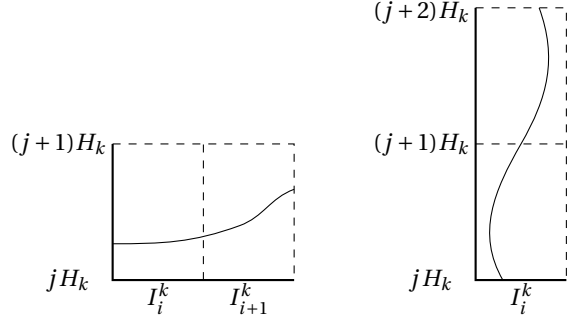
\begin{figure}[htb]
\centering
\begin{tikzpicture}[scale=0.6, every node/.style={scale=0.8}]
  \begin{scope}
    \draw[thick] (0,3) -- (0,0) -- (4,0);
    \draw[dashed] (0,3) -- (4,3) -- (4,0);
    \draw[dashed] (2,0) -- (2,3);
    \node at (1,-0.4) {$I_i^{k}$};
    \node at (3,-0.4) {$I_{i+1}^{k}$};
    \node at (-1,0) {$j H_k$};
    \node at (-1.2,3) {$(j+1)H_k$};
    \draw (0,0.8) to [out=0,in=-160] (2.8,1.2) to [out=20,in=200] (4,2);
  \end{scope}

  \begin{scope}[xshift=8cm]
    \draw[thick] (0,6) -- (0,0) -- (2,0);
    \draw[dashed] (0,6) -- (2,6) --  (2,0);
    \draw[dashed] (0,3) -- (2,3);
    \node at (1,-0.4) {$I_i^{k}$};
    \node at (-1,0) {$j H_k$};
    \node at (-1.2,3) {$(j+1)H_k$};
    \node at (-1.2,5.9) {$(j+2)H_k$};
    \draw (0.6,0) to [out=120,in=-120] (1,3) to [out=60,in=290] (1.4,6);
  \end{scope}
\end{tikzpicture}
\caption{Illustration of the events $C_{i, j}^k$ (left) and $D_{i, j}^k$ (right).}
\label{fig:CeD}
\end{figure}

For $k \geq 1$, each  $I_j^k$ can be partitioned into disjoint intervals from scale $k-1$:
\begin{equation}\label{firstandlast} I_j^k=\bigcup_{i \in l_{k,j}} I_i^{k-1},
\end{equation}
where the union runs over the set of indices
\begin{equation}
 l_{k,j} = \big\{ j\lfloor L_{k-1}^{\gamma-1}\rfloor, \cdots, (j+1)\lfloor L_{k-1}^{\gamma-1}\rfloor-1 \big\}.
 \end{equation}

Now fix an environment $\Lambda\subseteq \mathbb{Z}$.
Intervals will be recursively classified either good or bad as follows.

\begin{itemize}
\item at the bottom scale $k=0$, declare an interval bad if it does not intersect $\Lambda$ (see Figure~\ref{fig:bottom});
\item at higher scales $k \geq 1$, declare an interval bad if it contains two non-consecutive bad intervals from scale $k-1$ (see Figure~\ref{fig:GBintervals}). 
\end{itemize}

\begin{figure}
\centering
\begin{tikzpicture}[scale=0.4, every node/.style={scale=0.7}]
  \begin{scope}
    \draw[thin] (-1,0)--(20.7,0); 
    \node[below] at (1,-0.1) {Good};
    \node[below] at (5,-0.1) {Bad};
    \node[below] at (9,-0.1) {Good};
    \node[below] at (13,-0.1) {Good};
    \node[below] at (17,-0.1) {Good};
    \node[above] at (-1,0.3) {$0$};
    \node[above] at (19,0.3) {$L_{1}$};
    \foreach \y in {0,...,5} {\draw[thick] (4*\y-1,0.3) -- (4*\y-1,-0.3);}
    \foreach \z in {2,...,4} {\node[above] at (4*\z-1,0.3) {$\z L_{0}$};}
    \node[above] at (3,0.2) {$L_0$};
    \draw[thin, ->] (20.4,0) -- (20.7,0); 
    \foreach \x in {0,2,7,10,11,15,16,17,19} {
      \node[circle, fill=black, inner sep=2pt] at (\x,0) {};
    }
    \foreach \x in {-1,1,3,4,5,6,8,9,12,13,14,18,20} {
      \node[circle, draw=black, fill=white, inner sep=2pt] at (\x,0) {};
    }
  \end{scope}
\end{tikzpicture}
\caption{Bottom scale. 
Sites are colored black or white according to whether they belong to $\Lambda$ or not, respectively.
Good intervals contain at least one point of $\Lambda$.}
\label{fig:bottom}
\end{figure}
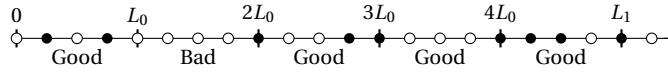

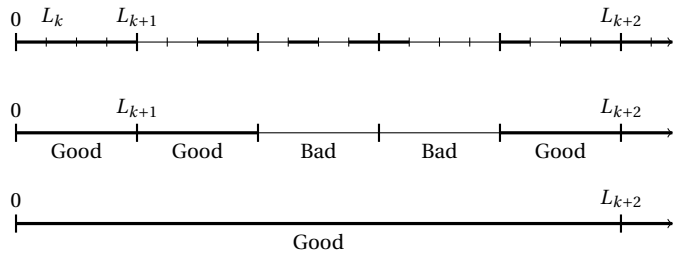
\begin{figure}
\centering
\begin{tikzpicture}[scale=0.4, every node/.style={scale=0.7}]
  \begin{scope}[yshift=6cm]
    \draw[thin] (-1,0)--(20.7,0); 
    \node[above] at (-1,0.3) {$0$};
    \node[above] at (3,0.3) {$L_{k+1}$};
    \foreach \y in {0,...,5} {\draw[thick] (4*\y-1,0.3) -- (4*\y-1,-0.3);}
    \foreach\x in {-1,0,1,...,20} {\draw[thin] (\x,0.15)--(\x,-0.15);}
    \node[above] at (0.2,0.3) {$L_k$};
    \node[above] at (19,0.3) {$L_{k+2}$};
    \draw[very thick] (-1,0) -- (3,0);
    \draw[very thick] (5,0) -- (7,0);
    \draw[very thick] (8,0) -- (9,0);
    \draw[very thick] (10,0) -- (12,0);
    \draw[very thick] (15,0) -- (16,0);
    \draw[very thick] (17,0) -- (20.7,0);
    \draw[thin, ->] (20,0) -- (20.7,0); 
  \end{scope}

  \begin{scope}[yshift=3cm]
    \draw[thin] (-1,0)--(20.7,0); 
    \node[below] at (1,-0.1) {Good};
    \node[below] at (5,-0.1) {Good};
    \node[below] at (9,-0.1) {Bad};
    \node[below] at (13,-0.1) {Bad};
    \node[below] at (17,-0.1) {Good};
    \node[above] at (-1,0.3) {$0$};
    \node[above] at (3,0.3) {$L_{k+1}$};
    \foreach \y in {0,...,5} {\draw[thick] (4*\y-1,0.3) -- (4*\y-1,-0.3);}
    \node[above] at (19,0.3) {$L_{k+2}$};
    \draw[very thick] (-1,0) -- (7,0);
    \draw[very thick] (15,0) -- (20.7,0);
    \draw[thin, ->] (20,0) -- (20.7,0); 
  \end{scope}

  \begin{scope}
    \draw[thin] (-1,0)--(20.7,0); 
    \node[below] at (9,-0.1) {Good};
    \node[above] at (-1,0.3) {$0$};
    \draw[thick] (-1,0.3) -- (-1,-0.3);
    \draw[thick] (19,0.3) -- (19,-0.3);
    \node[above] at (19,0.3) {$L_{k+2}$};
    \draw[very thick] (-1,0) -- (20.7,0);
    \draw[thin, ->] (20,0) -- (20.7,0);
  \end{scope}
\end{tikzpicture}
\caption{The configurations of good and bad intervals at scales $k$, $k+1$, and $k+2$. Good intervals are represented by thick solid lines, and bad intervals are represented by thin solid lines.}
\label{fig:GBintervals}
\end{figure}

Let $\rho$ be a random variable with distribution:
\begin{equation}
\label{e:stat_delay}
\mathbb{P}(\rho=k):=\dfrac{1}{\mathbb{E}(\xi)}\sum_{i=k+1}^{\infty}\mathbb{P}(\xi=i),\ \textrm{for every }k\in\mathbb{Z}_+,
\end{equation}
independent of everything else.
Roughly, $\rho$ is the delay that should be introduced to the renewal process $\Lambda \cap \mathbb{Z}_+$ in order to make it stationary.
In other words, modifying \eqref{e:def_renewal} to start with $x_0 = \rho$ instead of $x_0 = 0$, the one-sided renewal process $\Lambda_+ = \{x_0, x_1, x_2, \ldots\}$ is stationary.
Let us denote $\upsilon_\xi^{\rho,+}$ the law of that process.
This allows us to define
\begin{equation}
p_k:=\upsilon^{\rho,+}_\xi\big(\text{$I_0^k$ is bad}\big)=\upsilon^{\rho,+}_\xi\big(\text{$I_j^k$ is bad}\big),
\end{equation}
where the second equality follows from the stationarity of $\rho$.

Let us define, for every $i, j, k \in \mathbb{Z}_+$ and $p\in(0,1)$,
\begin{equation}
\label{eq:def_qk}
q_k(p; i,j):=\max\Bigg\{\underset{\substack{ \Lambda; \ I^k_i \textrm{ and }\\ I^k_{i+1}\textrm{are good}}}{\max}\mathbb{P}^\Lambda_p\big((C_{i,j}^k)^c\big), \underset{\substack{ \Lambda; \ I^k_i \textrm{ is} \\ \textrm{good}}}{\max} \mathbb{P}^\Lambda_p\big((D_{i, j}^k)^c\big) \Bigg\}.
\end{equation}
Translation invariance allows us to write, for every $k \in \mathbb{Z}_+$,
\begin{eqnarray}
\label{eq:qk_inv}
q_k(p)\vcentcolon=q_k(p;0,0) = q_k(p; i, j), \textrm{ for any }i, j \in \mathbb{Z}_+.
\end{eqnarray}

The next two lemmas, taken from \cite{Hilario2023} establish upper bounds for $p_k$ and $q_k(p)$.

\begin{lemma}[Lemma 3.1 in \cite{Hilario2023}]
\label{l:p_k}
For every $k\in\mathbb{Z}_+$ we have
\begin{eqnarray}
\label{e:ctrl_env} 
p_k\leq L_k^{-\varepsilon/2}.
\end{eqnarray}
\end{lemma}

\begin{lemma}[Lemma 3.2 in \cite{Hilario2023}]
\label{l:ctrl_cross}
There exist $c_3=c_3(L_0, \varepsilon)\in\mathbb{Z}_+$ and $p=p(L_0, \varepsilon, c_3)$ sufficiently close to 1 such that
\begin{equation}
q_k(p) \leq \exp(-L_k^{\beta}), \,\,\,\, \textrm{ for any }k\geq c_3,
\end{equation}
where the exponent is given by
\begin{equation}
\label{e:def_beta}
\beta = 1- \frac{\varepsilon}{8(\varepsilon+4)}.
\end{equation}
\end{lemma}

\begin{remark} 
That the value of $\beta$ given in \eqref{e:def_beta} works, follows from our choices of $\gamma$ and $\mu$ in \eqref{e:def_gamma_mu}, together with \cite[Eq.\ (37), pp.\ 3156]{Hilario2023}.
\end{remark}

Let us write $\upsilon_{\xi}^{n,+}$ for the one-sided renewal with interarrival distribution $\xi$ starting at position $n \geq 0$.

As a direct consequence of Lemma \ref{l:p_k} we have
\begin{lemma}
Let $\xi$ satisfy \eqref{e:xi_na_um_mais_epsilion} and \eqref{e:xi_aperiodic}.
Then, for every integer $n\geq 0$, we have
\begin{equation}
\label{e:eventually_good_n}
\upsilon_{\xi}^{n,+} \big(\text{$I^k_{0}$ and $I^k_1$ are good eventually} \big) =1.
\end{equation}
\end{lemma}

\begin{proof}
By \eqref{e:ctrl_env}, we have 
\[
\sum_{k=0}^{\infty} \upsilon^{\rho,+}_{\xi} \big(\text{$I^k_0$ or $I^k_1$ are bad}\big) \leq \sum_{k=0}^{\infty} L_k^{-\varepsilon/2} < \infty.
\]
Therefore,
\[
\upsilon^{\rho,+}_{\xi} \big(\text{$I^k_0$ and $I^k_1$ are good eventually}\big) =1.
\]
By \eqref{e:stat_delay}, $\rho$ assumes any positive integer value with positive probability.
Hence,
\[
\upsilon^{n,+}_{\xi} \big(\text{$I^k_0$ and $I^k_1$ are good eventually}\big) =\upsilon^{\rho,+}_{\xi} \big(\text{$I^k_0$ and $I^k_1$ are good eventually} \mid \rho = n\big) = 1.
\]
\end{proof}

\begin{lemma}

Let $\Lambda = \big\{\xi_k\big\}_{k\in \mathbb{Z}}$ where $\xi_k=|\mathcal{E}_{2k}|+1$.
Then, for almost every realization of $\Lambda$, we have
\begin{equation}
\text{$I^k_{-2}$, $I^k_{-1}$, $I^k_{0}$ and $I^k_1$ are good eventually}.
\end{equation}
\end{lemma}

\begin{proof}
Let us write $\xi$ for the law of $\xi_k$ with $k\neq 0$.
We also write $\bar{\upsilon}_{\xi}^{0}$ 
for the law of the two-sided 
renewal process similar to $\upsilon_{\xi}^{0}$, 
with the exception that the distribution of $\xi_0$ is that of $|\mathcal{E}_0|+1$.
Recall that  $\upsilon_{\xi}^{n,+}$ denotes the one-sided renewal with interarrival distribution $\xi$ starting at position $n \geq 0$. We have
\begin{equation}
\label{e:good_eventually_one_sided}
\begin{split}
\bar{\upsilon}_{\xi}^{0} & \big(\text{$I^k_{-2}$, $I^k_{-1}$, $I^k_{0}$ and $I^k_1$ are good eventually} \big) \\ 
& = \bar{\upsilon}_{\xi}^{0} \big(\text{$I^k_{-2}$ and $I^k_{-1}$ are good eventually} \big) \cdot \bar{\upsilon}_{\xi}^{0} \big(\text{$I^k_{0}$ and $I^k_{1}$ are good eventually} \big) \\
&\quad \quad =  {\upsilon}_{\xi}^{|\mathcal{E}_0|+1,+} \big(\text{$I^k_{0}$ and $I^k_{1}$ are good eventually} \big) \cdot {\upsilon}_{\xi}^{0,+} \big(\text{$I^k_{0}$ and $I^k_{1}$ are good eventually} \big),
\end{split}
\end{equation}
where we have used independence of what occurs on both sides of the origin and  reflection invariance of the one-sided process.

By \eqref{e:eventually_good_n} with $n=0$, $\upsilon_{\xi}^{0,+} \big(\text{$I^k_{0}$ and $I^k_{1}$ are good eventually} \big) =1$, so it suffices to prove that 
\[
 {\upsilon}_{\xi}^{|\mathcal{E}_0|+1,+} \big(\text{$I^k_{0}$ and $I^k_{1}$ are good eventually} \big) =1.
\]

In fact we have,
\begin{equation}
\begin{split}
{\upsilon}_{\xi}^{|\mathcal{E}_0|+1,+} & \big(\text{$I^k_{0}$ and $I^k_{1}$ are good eventually} \big) \\ 
& = \sum_{n=0}^{\infty} {\upsilon}_{\xi}^{|\mathcal{E}_0|+1,+} \left(\text{$I^k_{0}$ and $I^k_{1}$ are good eventually} ~ \big\vert ~ |\mathcal{E}_0|+1 = n \right) \cdot \mathbb{P}(|\mathcal{E}_0|+1 = n) \\
& =  \sum_{n=0}^{\infty} {\upsilon}_{\xi}^{n,+} \big(\text{$I^k_{0}$ and $I^k_{1}$ are good eventually} \big) \cdot \mathbb{P}(|\mathcal{E}_0|+1 =n) = 1,
\end{split}
\end{equation}
where we have also used \eqref{e:eventually_good_n} at the last step.
\end{proof}

Let us now consider the events
\begin{equation}
\label{e:circuit_scale_k}
F_{k}=\Big(\bigcap_{i={-2}}^{0} C^{k}_{i,1}\Big) \cap \Big(\bigcap_{i={-2}}^{1} D^{k}_{i,0}\Big).
\end{equation}
See Figure \ref{f:circuit_scale_k}.

\begin{figure}[htb!]
\centering
\begin{tikzpicture}[scale=0.6, every node/.style={scale=0.8}]
  \begin{scope}[yshift=3cm]
    \draw[thick] (0,3) -- (0,0) -- (4,0);
    \draw[dashed] (0,3) -- (4,3) -- (4,0);
    \draw[dashed] (2,0) -- (2,3);
    \draw (0,.8) to [out=0,in=-160] (2.8,1.2) to [out=20,in=200] (4,2);
  \end{scope}
  \begin{scope}[yshift=3cm, xshift=2cm]
    \draw[thick] (0,3) -- (0,0) -- (4,0);
    \draw[dashed] (0,3) -- (4,3) -- (4,0);
    \draw[dashed] (2,0) -- (2,3);
    \draw (0,.8) to [out=0,in=-160] (2.8,1.2) to [out=20,in=200] (4,2);
  \end{scope}
  \begin{scope}[yshift=3cm, xshift=4cm]
    \draw[thick] (0,3) -- (0,0) -- (4,0);
    \draw[dashed] (0,3) -- (4,3) -- (4,0);
    \draw[dashed] (2,0) -- (2,3);
    \draw (0,.8) to [out=0,in=-160] (2.8,1.2) to [out=20,in=200] (4,2);
  \end{scope}

  \begin{scope}
    \draw[thick] (0,6) -- (0,0) -- (2,0);
    \draw[dashed] (0,6) -- (2,6) --  (2,0);
    \draw[dashed] (0,3) -- (2,3);
    \node at (1,-0.4) {$I_{-2}^{k}$};
    \node at (-1,0) {$0$};
    \node at (-1.2,3) {$H_k$};
    \node at (-1.2,5.9) {$2H_k$};
    \node at (3,-0.4) {$I_{-1}^{k}$};
    \draw (0.6,0) to [out=120,in=-120] (1,3) to [out=60,in=290] (1.4,6);
  \end{scope}
  \begin{scope}[xshift=2cm]
    \draw[thick] (0,6) -- (0,0) -- (2,0);
    \draw[dashed] (0,6) -- (2,6) --  (2,0);
    \draw[dashed] (0,3) -- (2,3);
    \draw (1,3) to [out=60,in=290] (1.4,6);
    \draw[thin, dotted] (0.6,0) to [out=120,in=-120] (1,3);
  \end{scope}
  \begin{scope}[xshift=4cm]
    \draw[thick] (0,6) -- (0,0) -- (2,0);
    \draw[dashed] (0,6) -- (2,6) --  (2,0);
    \draw[dashed] (0,3) -- (2,3);
    \node at (1,-0.4) {$I_{0}^{k}$};
    \draw (1,3) to [out=60,in=290] (1.4,6);
    \draw[thin, dotted] (0.6,0) to [out=120,in=-120] (1,3);
  \end{scope}\begin{scope}[xshift=6cm]
    \draw[thick] (0,6) -- (0,0) -- (2,0);
    \draw[dashed] (0,6) -- (2,6) --  (2,0);
    \draw[dashed] (0,3) -- (2,3);
    \node at (1,-0.4) {$I_{1}^{k}$};
    \draw (0.6,0) to [out=120,in=-120] (1,3) to [out=60,in=290] (1.4,6);
  \end{scope}
  
\end{tikzpicture}
\caption{ Occurrence of the crossing events $C_{i, 1}^k$, $i=-2,-1,0$ together with $D_{i, 0}^k$ $i=-2,-1,0,1$ entails the existence of a dual circuit about the origin.}
\label{f:circuit_scale_k}
\end{figure}
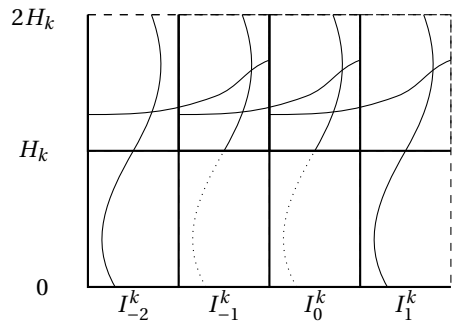

\begin{lemma}\label{l:Fk}
Let $\Lambda = \big\{|\mathcal{E}_{2k}|+1\big\}_{k\in \mathbb{Z}}$.
Then, for every $p$ sufficiently large,
\begin{equation}
\mathbb{P}^{\Lambda}_p (F_k \text{ eventually})=1,
\label{e:cross_eventually}
\end{equation}
for almost all realizations of $\Lambda$.
\end{lemma}

\begin{proof}

Conditional on a realization of $\Lambda$ for which $I^k_{-2}$, $I^k_{-1}$, $I^k_{0}$ and $I^k_1$ are good eventually 
Lemma \ref{l:ctrl_cross} yields
\begin{equation}
\sum_{k\geq k_0} \mathbb P_{p}^{\Lambda} (F_k^c) \leq \sum_{k\geq k_0} 7 \exp(-L_k^\beta) < \infty.
\end{equation}
Hence, a Borel-Cantelli argument implies \eqref{e:cross_eventually}.

\end{proof}

\section{Proof of Theorem~\ref{theo1}}

We are now in a position to prove Theorem~\ref{theo1}. 
As argued in the last paragraph of Section~\ref{s:renormalization}, it suffices to show that for a sufficiently small infection rate $\lambda$, there almost surely exists a semi-circuit of good vertices enclosing the origin in the renormalized lattice $\mathbb{Z} \times \mathbb{Z}_+$. 
Lemma~\ref{l:Fk} enables us to study the existence of such circuits in a different bond percolation process.
We then apply Lemma~\ref{l:compare_site_bond} to compare the probabilities of crossing events, and thus of existence of semi-circuits, in both models. 
This lemma establishes that the probability of semi-circuits in the former percolation process is bounded below by the probability in the latter, with a certian parameter $\rho$ that can be made as large as we want by decreasing $\lambda$. 
As stated in Lemma~\ref{l:Fk}, this lower bound equals 1, provided that $\rho$ is sufficiently large.
The result follows.

\end{document}